\documentclass[12pt]{amsart}
\usepackage{amsmath,bbm, color}
\usepackage{latexsym,mathrsfs,bm}
\usepackage{hyperref}

\usepackage{url}
\usepackage[english]{babel}
\usepackage{paralist}
\usepackage{amsthm,amsfonts,amssymb,amsmath}
\usepackage{comment}

\usepackage[latin1]{inputenc}

\newcommand{\tRe}{\textup{Re }}

\newcommand{\sumh}{\sideset{}{^h}\sum}
\newcommand{\sumd}{\sideset{}{^d}\sum}

\newcommand{\N}{\mathcal{N}}

\newcommand{\es}[1]{\begin{equation}\begin{split}#1\end{split}\end{equation}}
\newcommand{\est}[1]{\begin{equation*}\begin{split}#1\end{split}\end{equation*}}

\newcommand{\M}{\mathcal{M}}

\renewcommand{\mod}[1]{~\pr{\textnormal{mod}~#1}}
\newtheorem{thm}{Theorem}[section]

\newtheorem{prop}[thm]{Proposition}
\newtheorem{lem}[thm]{Lemma}
\newtheorem{lemma}[thm]{Lemma}

\theoremstyle{remark}

\newtheorem{rem*}{Remark}
\newcommand{\pr}[1]{\left( #1\right)}

\newcommand{\e}[1]{\operatorname{e}\pr{ #1}}

\newcommand{\bfrac}[2]{\left(\frac{#1}{#2}\right)}

\newcommand{\lt}{\left}
\newcommand{\rt}{\right}

\newcommand{\bsm}{\lt(\begin{smallmatrix}}
\newcommand{\esm}{\end{smallmatrix}\rt)}

\newcommand{\sym}{\textup{sym}}

\newcommand{\Dis}{\textup{Dis}}
\newcommand{\Ctn}{\textup{Ctn}}
\newcommand{\Hol}{\textup{Hol}}

\def\sumh{\operatornamewithlimits{\sum\nolimits^h}}

\newcommand{\sumthree}{\operatorname*{\sum\sum\sum}}

\let\originalleft\left
\let\originalright\right
\renewcommand{\left}{\mathopen{}\mathclose\bgroup\originalleft}
\renewcommand{\right}{\aftergroup\egroup\originalright}

\numberwithin{equation}{section}

\newcommand{\W}{\mathcal W}
\newcommand{\Rep}{\textrm{Re}}
\newcommand{\mL}{\mathcal L}
\newcommand{\fm}{\mathfrak m}

\begin{document}
\title[Low-lying zeros of a large orthogonal family]{Low-lying zeros of a large orthogonal family of automorphic $L$-functions}

\date{
\today}

\author[S. Baluyot]{Siegfred Baluyot}
\address{American Institute of Mathematics\\
San Jose, CA 95112}
\email{{sbaluyot@aimath.org}}

\author[V. Chandee]{Vorrapan Chandee}
\address{Mathematics Department \\ Kansas State University \\ Manhattan, KS 66503}
\email{chandee@ksu.edu}

\author[X. Li]{Xiannan Li}
\address{Mathematics Department \\ Kansas State University \\ Manhattan, KS 66503}
\email{xiannan@ksu.edu }


\allowdisplaybreaks
\numberwithin{equation}{section}
\begin{abstract}
We study a new orthogonal family of $L$-functions associated with holomorphic Hecke newforms of level $q$, averaged over $q \asymp Q$.  To illustrate our methods, we prove a one level density result for this family with the support of the Fourier transform of the test function being extended to be inside $(-4, 4)$.  The main techniques developed in this paper will be useful in developing further results for this family, including estimates for high moments, information on the vertical distribution of zeros, as well as critical line theorems.
\end{abstract}

\maketitle

\section{Introduction} 

A major theme in analytic number theory is to study statistics over families of $L$-functions.  Here, a foundational heuristic is that the statistics of families of $L$-functions should match analogous statistics from classical compact groups of random matrices.  Katz and Sarnak proved that such heuristics hold in many examples of zeta and $L$-functions over function fields \cite{KaSa}, and suggest that analogous results should hold over number fields.  
In this paper, we are interested in studying the low lying zeros of $L$-functions near the central point.  Here, the Katz-Sarnak \cite{KaSa} philosophy predicts that the distribution of the low-lying zeros is governed by some underlying symmetry group.  The full predictions of these heuristics remain far out of reach; however, there are several examples of families of $L$-functions for which conjectures have been verified in certain restricted ranges.

One example of such a family is the family of $L$-functions associated with holomorphic Hecke eigencuspforms of fixed weight $k$ and square-free level $q$. For this family, assuming the Generalized Riemann Hypothesis (GRH), Iwaniec, Luo, and Sarnak~\cite{ILS} verified that the density predicted by Katz and Sarnak for the low-lying zeros holds true as $q\rightarrow \infty$ provided that the support of the Fourier transform of the test function is compact and contained in the interval $(-2,2)$.  The square-free condition on the level $q$ was removed in later work of Barrett et al. \cite{BBDDM}.

In this paper, we develop a novel method that allows us to prove a stronger result for an extension of this family of $L$-functions. More precisely, we study the orthogonal family of $L$-functions associated with holomorphic Hecke newforms of level $q$, averaged over {\it all } $q \asymp Q$.  Under GRH, we show that the density predicted by Katz and Sarnak holds true for the extended family when the support of the Fourier transform of the test function is contained in the interval $(-4,4)$. The technique we develop has applications to the study of other statistics of this family, such as the sixth and eighth moments, the number of zeros on the critical line, and the vertical distribution of zeros.

To state our results more precisely, let $S_k(q)$ be the space of cusp forms of fixed weight $k \ge 3$ for the group $\Gamma_0(q)$ with trivial nebentypus,  where as usual,
$$
\Gamma_0(q) := \left\{ \left( \left.{\begin{array}{cc}
   a & b \\
   c & d \\
  \end{array} } \right)  \ \right|  \  ad- bc = 1 , \ \  c \equiv 0 \mod q \right\}.
$$
Let $\mathcal H_k(q) \subset  S_k(q)$ be an orthogonal basis of the space of newforms consisting of Hecke cusp newforms, normalized so that the first Fourier coefficient is $1$. For each $f \in \mathcal H_k(q) $, we let $L(s, f)$ be the $L$-function associated to $f$, defined for Re$(s) > 1$ by
\es{ \label{def:Lsf}
L(s, f) = \sum_{n \geq 1} \frac {\lambda_f(n)}{n^{s}} &= \prod_p \pr{1 - \frac{\lambda_f(p)}{p^s} + \frac{\chi_0(p)}{p^{2s}}}^{-1} \\
&= \prod_p \pr{1 - \frac{\alpha_f(p)}{p^s} }^{-1} \pr{1 - \frac{\beta_f(p)}{p^s} }^{-1},}
where $\{\lambda_f(n)\}$ is the set of Hecke eigenvalues of $f$ and $\chi_0$ denotes the trivial character $\bmod q$. Since $f$ is a newform, $L(s, f)$ can be analytically continued to the entire complex plane and satisfies the functional equation
\es{ \label{eqn:fncL} \Lambda\pr{\tfrac 12 + s, f} = i^k \overline{\eta}_f\Lambda\pr{ \tfrac 12 - s, f} }
where the completed $L$-function $\Lambda(s, f)$ is defined by
\est{\Lambda\pr{ \tfrac 12 + s, f} = \pr{\frac q{4\pi^2}}^{\frac s2} \Gamma \pr{s + \frac k2} L\pr{ \tfrac 12 + s, f},}
and the sign of the functional equation is $i^k \eta_f = \pm 1$.  For convenience, we normalize our sum over $f$ to play well with spectral theory. Suppose for each $f \in \mathcal H_k(q)$, we have an associated number $\alpha_f$. Then we define the harmonic average of $\alpha_f$ over $\mathcal H_k(q)$ to be 
\begin{equation}\label{eqn: harmonicweightsummation}
\sumh_{f \in \mathcal H_k(q)} \alpha_f = \frac{\Gamma(k-1)}{(4\pi)^{k-1}}\sum_{f \in \mathcal H_k(q)} \frac{\alpha_f}{\|f\|^2},
\end{equation}
where  $\|f \|^2 = \int_{\Gamma_0(q) \backslash \mathbb H} |f(z)|^2 y^{k-2} \> dx \> dy. $ 

Now let $\Phi(x)$ be an even Schwartz class function, and let 
$$
\widehat \Phi(t) = \int_{-\infty}^{\infty} \Phi(x) e^{-2\pi i x t} \> dx
$$
be the usual Fourier transform.

The \textit{Density Conjecture} from the Katz-Sarnak philosophy ~\cite{KaSa} predicts that
\begin{equation}\label{eqn:katzsarnakprediction}
\lim_{q\rightarrow \infty} \frac{1}{\#\mathcal{H}_k(q) } \sum_{f\in \mathcal{H}_k(q)} \sum_{\gamma_f} \Phi \left( \frac{\gamma_f }{2 \pi} \log q\right) = \int_{-\infty}^{\infty} \Phi(x) \left( 1 +\frac{1}{2}\delta_0(x)\right)\,dx,
\end{equation}
where $\gamma_f$ runs through the imaginary parts of the nontrivial zeros of $L(s,f)$, each repeated according to multiplicity, and $\delta_0(x)$ is the Dirac delta distribution at $x=0$. Here $1 + \delta_0(x)/2$ is the kernel associated with orthogonal symmetry. Assuming GRH for certain $L$-functions, Iwaniec, Luo, and Sarnak~\cite{ILS} prove that \eqref{eqn:katzsarnakprediction} holds when $\widehat \Phi$ is compactly supported in $(-2,2)$ and the variable $q$ runs through the square-free integers.

In this paper, we study a larger family of orthogonal $GL(2)$ $L$-functions by including an average over $q$.  To this end, we fix a smooth function $\Psi$ compactly supported on $\mathbb{R}_{>0}$ with $\widehat{\Psi}(0) \neq 0$ and let
$$
\mathscr{OL}(Q) := \frac{1}{N(Q)}\sum_q \Psi\bfrac{q}{Q} \sumh_{f \in \mathcal H_{k}(q)} \sum_{\gamma_f} \Phi \left( \frac{\gamma_f }{2 \pi} \log q\right),
$$
where 
$$
N(Q) := \sum_q \Psi\bfrac{q}{Q} \sumh_{f \in \mathcal H_{k}(q)} 1.
$$
Note that by work of G. Martin \cite{GMar}, we have
\begin{equation}\label{eqn:asympN(q)}
N(Q) \asymp \sum_q \Psi\bfrac{q}{Q} \asymp  Q \widehat{\Psi}(0),
\end{equation}since $\widehat{\Psi}(0) \neq 0$.  Our main result is as follows.

\begin{thm}\label{thm:main} Assume GRH. Let $\Phi$ be an even Schwartz function with $\widehat \Phi$ compactly supported in $(-4, 4).$ Then with notation as before,
$$ \lim_{Q \rightarrow \infty} \mathscr{OL}(Q) = \int_{-\infty}^{\infty} \Phi(x) \left( 1 +\frac{1}{2}\delta_0(x)\right)\,dx,
$$
where $\delta_0(x)$ is the Dirac $\delta$ distribution at $x=0$.
\end{thm}


The size of the support of $\widehat{\Phi}$ is doubled in our result compared to the result of Iwaniec, Luo and Sarnak \cite{ILS} because we have taken full advantage of the additional average over all levels $q$ around size $Q$.  This family of around $Q^2$ forms has elements with conductors of size around $Q$.  It is instructive to compare this with results on the similar family of Dirichlet $L$-functions attached to Dirichlet characters $\chi \bmod q$ with $q$ around size $Q$.

There, Conrey, Iwaniec, and Soundararajan~\cite{CIS} developed an asymptotic large sieve. Their work forms the foundation for later sixth and eighth moment calculations \cite{CIS2,CL,CLMR}, critical line theorems \cite{CIS3}, and other powerful results, and we expect the framework developed here to lead to analogous results for our family of cusp forms.

There are fundamental differences between these two families, and here we point out one that is especially relevant in the context of one level density. Simply put, when averaging over $m, n$ and modulus $q \asymp Q$, the asymptotic large sieve from~\cite{CIS} depends critically on a ``complementary divisor trick'' that allows us to switch the modulus to $ \frac{m\pm n}{Q}$.  In our family of cusp forms, we develop a complementary level trick, which allows us to switch to level $\frac{\sqrt{mn}}{Q}$.  The geometric mean $\sqrt{mn}$ is the same as the arithmetic mean $\frac{m+n}{2}$ if and only if $m=n$, and is far smaller when $m$ and $n$ are far apart.  This fact plays a critical role here, since in the context of one level density, we essentially have $m \asymp Q^4$ and $n=1$.

In terms of one level density results, Fiorilli and Miller~\cite{FM} studied the family of Dirichlet $L$-functions over all characters modulo $q$, where $q \asymp Q$ and wanted to understand
$$
\sum_{q \asymp Q} \sum_{\chi \bmod q } \sum_{\gamma_f} \Phi\left( \frac{\gamma_f}{2\pi} \log q\right).
$$
Previously, the work of Hughes and Rudnick \cite{HR} studided the analogous quantity over the smaller family of Dirichlet $L$-functions attached to characters modulo $q$ with fixed prime $q$, and verified the analogous density conjecture  with the support of $\widehat{\Phi}$ restricted within $[-2, 2]$.  The additional average over $q$ in the work of Fiorilli and Miller \cite{FM} does not lead directly to an improvement on the size of the support of $\widehat{\Phi}$.  Rather, Fiorilli and Miller \cite{FM} were able to extend the support to $(-4, 4)$ assuming both GRH and a very strong ``de-averaging hypothesis'' on the variance of primes in residue classes. 

Recently, Drappeau, Pratt and Radziwi\l\l \, \cite{DPR} considered the one level density for the similar large family of Dirichlet $L$-functions attached to all primitive characters modulo $q$ for all $q \asymp Q$ and showed that the support of $\widehat \Phi$ can be extended to be within $(-2 - 50/1093, 2 + 50/1093)$ unconditionally.  It remains a challenging problem to extend the support to something closer to $(-4, 4)$ without additional hypotheses.

Our main Theorem~\ref{thm:main} will follow from the explicit formula and the Proposition below. The proof of Theorem~\ref{thm:main} will be in Section~\ref{sec:proofofMainTheorem}.

{\prop\label{prop:main} 
Assume GRH. Let $\Phi$ be an even Schwartz function with $\widehat \Phi$ compactly supported in $(-4, 4).$ Moreover, let 
\es{\label{def:cfn}
c_f(n) = \left\{ \begin{array}{cc}
     \alpha_f(p)^\ell + \beta_f(p)^\ell  & \ \textrm{if} \  n = p^\ell \\
     0 & \ \textrm{otherwise}, 
\end{array}\right. }
where $\alpha_f$ and $\beta_f$ are as in \eqref{def:Lsf}. Then 
$$ \lim_{Q \rightarrow \infty }\frac{1}{N(Q)}  \sum_q \Psi\bfrac{q}{Q} \frac{1}{\log q} \sumh_{f \in \mathcal{H}_k(q)} \sum_n \frac{\Lambda(n)c_f(n)}{\sqrt{n}} \widehat \Phi\left( \frac{\log n}{\log q}\right) = -\frac{ \Phi(0)}{4}.  $$
 Here $\Lambda(n)$ is the usual von Mangoldt function, which is $\log p$ when $n = p^\ell$ and $0$ otherwise.
}

\subsection{Outline of the proof}

To prove Proposition~\ref{prop:main}, we use Petersson's formula to reduce the problem to morally bounding a sum like
\begin{align}\label{eqn:simplesum1}
\frac{1}{N(Q)}  \sum_q \Psi\bfrac{q}{Q}   \sum_p \frac{\log p }{\sqrt{p}} \widehat \Phi\left( \frac{\log p}{\log q}\right) \sum_{c\ge 1} S(p, 1; cq) J_{k-1}\bfrac{4 \pi \sqrt{p}}{cq}.  
\end{align}
Here, we have suppressed technical complications arising from the average $\sumh_{f \in \mathcal H_{k}(q)}$ being a sum over primitive forms, rather than an orthogonal basis of all forms.  This requires pruning, and we suppress the details here for simplicity.  The sum in \eqref{eqn:simplesum1} is easy to bound if $\widehat{\Phi}$ is supported on $(-2, 2)$ by an application of the additive Large Sieve.  

We want to understand \eqref{eqn:simplesum1} when $\widehat{\Phi}$ has support extended to $(-4, 4)$.  To do this, we need to extract non-trivial cancellation in the sum over Kloosterman sums.  For simplicity, we restrict our attention to the transition region of the Bessel function, so that $cq \asymp \sqrt{p}$. The support of $\widehat{\Phi}$ is compact and contained in $(-4, 4)$, so the sum over $p$ is restricted to $p\le Q^{4-\delta}$ for some $\delta > 0$ depending on $\Phi$.  Hence, we see that in this range
$$c \ll Q^{1-\delta/2}.
$$

This motivates us to write \eqref{eqn:simplesum1} as
\begin{align*}
\frac{1}{N(Q)}  \sum_p \frac{\log p }{\sqrt{p}} \sum_{c\ge 1}  \mathcal{S}(p, c)
\end{align*}
where
\begin{equation*}
    \mathcal{S}(p, c) = \sum_q  S(p, 1; cq) \Psi\bfrac{q}{Q}\widehat\Phi\left( \frac{\log p}{\log q}\right) J_{k-1}\bfrac{4 \pi \sqrt{p}}{cq}
\end{equation*}
can be transformed into a sum of forms of level $c$ via Kuznetsov's formula. In Proposition~\ref{prop:main}, we started with a sum over forms of level $q\asymp Q$. Thus, in applying Kuznetsov's formula, we have succeeded in lowering the conductor from $q \asymp Q$ to $c \ll Q^{1-\delta/2}$.

In most cases, we can use an appropriate choice of basis and GRH to bound the resulting sums over $p$.  To be more precise, we have morally that
$$ \sum_{p\asymp P} \frac{\log p }{\sqrt{p}} \lambda_g(p) \ll \log^2 Q
$$where $\lambda_g$ is some type of Hecke eigenvalue associated to a holomorphic form, a Maass form, or an Eisenstein series, and $P \asymp Q^{4-\delta}$.  For holomorphic forms and Maass forms, we reduce the problem to bounding these sums by choosing a special basis based on Atkin-Lehner theory, while for Eisenstein series, we use the relatively recent explicit calculations of Kiral and Young \cite{KY}.

An exception occurs in the case of an Eisenstein series, where one may morally replace $\lambda_g(p)$ with sums involving $\chi(p)$ for certain Dirichlet characters $\chi$.  When $\chi$ is non-principal, the above approach holds. However, when $\chi = \chi_0$ is principal, the sum over $p$ is genuinely large and is essentially $P^{1/2-it}\tilde{V}_1(1/2-it)$ for some rapidly decaying function $\tilde{V}_1$, where $t$ is the spectral parameter and $P$ can be as large as $Q^{4-\delta}$. We get around this issue by taking advantage of the average over the spectral parameter. This is in contrast to other situations, e.g. \cite{DFI}, where a truly large contribution arises from the Eisenstein series.

In \S\ref{sec:prelimresults} and \S \ref{sec:kutz}, we gather useful standard results for the rest of the paper.  We apply some initial steps in \S\ref{sec:mainsetup} and reduce the proof of Proposition~\ref{prop:main} and Theorem~\ref{thm:main} to Proposition~\ref{prop:boundSigma_1}.  In \S\ref{sec:applypetersson}, we apply Ng's~\cite{Ng} version of the Petersson's formula for primitive forms, and prune our sum by truncating the sum over less important parameters. This sets the stage for our application of Kutznetsov in \S\ref{sec:applykutz}, which reduces our task to bounding the contribution of the holomorphic forms and discrete spectrum in Proposition~\ref{prop:DisHol}, and the contribution of the continuous spectrum in Proposition~\ref{prop:Ctn}.  The proof of Proposition~\ref{prop:DisHol} is carried out in \S\ref{sec:dispart}.
The proof of Proposition~\ref{prop:Ctn} is carried out in \S\ref{sec:ctn}, where the contribution of the non-principal characters resembles the treatment of the Maass form contribution, while the principal character contribution requires different methods.


\section{Notation and Preliminary Results} \label{sec:prelimresults}

Throughout the paper, we follow the standard convention in analytic number theory of letting $\epsilon$ denote an arbitrarily small positive real
number whose value may change from line to line. We always use $p$ to denote a prime number. 

In this section, we collect some lemmas we will need in our proofs. 

\subsection{Approximate orthogonality}
We begin by stating the orthogonality relations for our family. These are the standard Petersson's formula (e.g. see \cite{Iwaniec}), and a version of Petersson's formula that is restricted to newforms and is due to Ng~\cite{Ng}.

Recall that $S_k(q)$ is the space of cusp forms of weight $k$ and level $q$. Let $B_k(q)$ be any orthogonal basis of $S_k(q)$. Define
\begin{equation*}
\Delta_q(m,n ) = \Delta_{k, q}(m, n) = \sumh_{f\in B_k(q)}\lambda_f(m)\lambda_f(n),
\end{equation*}
where the summation symbol $\sumh$ means we are summing with the same weights found in \eqref{eqn: harmonicweightsummation}. The usual Petersson's formula (e.g. see \cite{Iwaniec}) is the following.
\begin{lemma}\label{lem:usualPetersson}
If $m,n,q$ are positive integers, then
$$ \Delta_q(m, n) = \delta(m, n)+ 2\pi i^{-k} \sum_{c\geq 1} \frac{S(m, n;cq)}{cq} J_{k-1}\bfrac{4\pi \sqrt{mn}}{cq},
$$
where $\delta(m,n)=1$ if $m=n$ and is $0$ otherwise, $S(m,n;cq)$ is the usual Kloosterman sum, and $J_{k-1}$ is the Bessel function of the first kind.
\end{lemma}

Lemma~\ref{lem:usualPetersson}, the Weil bound for Kloosterman sums, and standard facts about the Bessel function imply the following Lemma (see Corollary 2.2 in \cite{ILS}).
\begin{lemma}\label{lem:petertruncate}
If $m,n,q$ are positive integers, then
\begin{equation*}
\Delta_q(m, n) = \delta(m, n) + O\left(\frac{\tau(q) (m, n, q)(mn)^{\epsilon}}{q ((m, q)+(n, q))^{1/2}} \bfrac{mn}{\sqrt{mn} + q}^{1/2} \right),
\end{equation*}
where $\tau(q)$ is the divisor function and $\delta(m,n)=1$ if $m=n$ and is $0$ otherwise.
\end{lemma}

For our purposes, we need to isolate the newforms of level $q$.  To be precise, recall that $\mathcal H_k(q)$ is the set of newforms of weight $k$ and level $q$ which are also Hecke eigenforms. We need a formula for
\begin{equation*}
\Delta_q^*(m, n) := \sumh_{f\in \mathcal H_k(q)} \lambda_f(m)\lambda_f(n).
\end{equation*}
A formula is known for squarefree level $q$ due to Iwaniec, Luo and Sarnak~\cite{ILS}, and for $q$ a prime power due to Rouymi~\cite{rouymi}. These formulas have been generalized to all levels $q$ by Ng \cite{Ng} (see also the works of Barret et al.\cite{BBDDM}, and Petrow \cite{Pe}). Ng's Theorem 3.3.1 contains some minor typos, but the corrected version is as follows.
\begin{lemma}\label{lem:PeterssonNg}
Suppose that $m,n,q$ are positive integers such that $(mn,q)=1$, and let $q = q_1q_2$, where $q_1$ is the largest factor of $q$ satisfying $p|q_1 \Leftrightarrow p^2|q$. Then
\begin{align*}
\Delta_q^*(m,n) = \sum_{\substack{q=L_1L_2d \\ L_1|q_1 \\ L_2|q_2}} \frac{\mu(L_1L_2)}{L_1L_2}  \prod_{\substack{p|L_1 \\ p^2 \nmid d}}  \left( 1-\frac{1}{p^2} \right)^{-1} \sum_{e|L_2^{\infty} }\frac{\Delta_d(m,ne^2)}{e}.
\end{align*}
\end{lemma}

\subsubsection*{Remark}
Due to the factor $\mu(L_1L_2),$ we may assume that $L_1$ and $L_2$ are square-free, and $(L_1, L_2) = 1.$ Furthermore, the condition that $L_1|q_1$ and $L_2|q_2$ is equivalent to the condition that $L_1|d$ and $(L_2,d)=1$. To see this, suppose first that $L_1|q_1$ and $L_2|q_2$.  If $p|L_1$, then $p^2|q$. So $p$ must divide $d$ since $q=L_1L_2d$ and $p \nmid L_2$. This concludes that $L_1 | d$ since $L_1$ is square-free.  If $p|d$ and $p|L_2$, then $p^2 | q$, and so $p | q_1$. But $(q_1, q_2) = 1$, so $p = 1.$ Thus we must have $(L_2,d)=1$. On the other hand, suppose that $L_1|d$ and $(L_2,d)=1$. If $p|L_1$, then $p|d$, and so $p^2| q$. Thus $p|q_1$. This implies that $L_1|q_1$. If $p|L_2$ then $p\nmid L_1d$ since $(L_2,dL_1)=1$. Hence the exponent of $p$ in the prime factorization of $q=L_1L_2d$ is exactly $1$, so $p|q_2$ by definition of $q_2$. Thus $L_2|q_2$.

The next lemma collects some well known properties and formulas for the $J$-Bessel function.
{\lem\label{jbessel} Let $J_{k-1}$ be the $J$-Bessel function of order $k-1$. We have
\est{ J_{k-1}(2\pi x) =\frac{1}{2\pi \sqrt{x}}\left(W_k(2\pi x) \e{x-\frac{k}{4}+\frac{1}{8}} + \overline{W}_k(2\pi x) \e{-x+\frac{k}{4}-\frac{1}{8} }\right),} where $W_k^{(j)}(x)\ll_{j,k} x^{-j}.$ Moreover, 
\begin{equation*}  J_{k-1}(2x) =\sum _{\ell = 0}^{\infty} (-1)^{\ell} \frac{x^{2\ell+k-1}}{\ell! (\ell+k-1)!}
\end{equation*}
and $$J_{k-1}(x)\ll \textup{min} \{x^{-1/2}, x^{k-1}\}.$$
}

The proof of the first three claims of Lemma~\ref{jbessel} can be found in \cite[p. 206]{Watt}, and the statement of the last claim is modified from Equation 16 of Table 17.43 in \cite{GR}.  
\vskip 0.1in

\subsection{An explicit formula and some consequences of GRH}

The first step in our proof of Theorem~\ref{thm:main} is to use the following explicit formula, which relates zeros of $L(s,f)$ with a sum over its prime power coefficients.

\begin{lem}\label{lem:Explicit} Let $\Phi$ be an even Schwartz function whose Fourier transform has compact support. We have
\begin{align*}
    \sum_{\gamma_f} \Phi \left(\frac{\gamma_f}{2\pi} \log q\right)  &= -\frac{1}{\log q}\sum_{n=1}^{\infty}\frac{\Lambda(n)[c_f(n) + c_{\bar f}(n)]}{\sqrt{n}} \widehat\Phi\left(\frac{\log n}{\log q}\right)+ \int_{-\infty}^{\infty} \Phi(x) \> dx +O_k\left(\frac{1}{\log q} \right),
\end{align*}
where $c_f(n)$ is defined in \eqref{def:cfn}.
\end{lem}

\begin{proof} Since $\widehat{\Phi}$ has compact support, we have that $\Phi$ is an entire function. Consider the integral
$$
\frac{1}{2\pi i}\int_{(1)}\frac{L'}{L}\left(s + \frac 12,f\right) \Phi\left(-is \frac{\log q}{2\pi}\right)\, ds.
$$
We move the line of integration to the left, passing poles at the zeros of $L(s, f)$, and we deduce that
\es{\label{eq:logintegral}
    \frac{1}{2\pi i}\int_{(1)}\frac{L'}{L}\left(s + \frac 12,f\right)\Phi\left(-is \frac{\log q}{2\pi}\right)\,ds = &\frac{1}{2\pi i }\int_{(-1)}\frac{L'}{L}\left(s + \frac 12,f\right)\Phi\left(-is \frac{\log q}{2\pi}\right)\,ds \\
     &+ \sum_{\gamma_f} \Phi \left(\frac{\gamma_f}{2\pi} \log q\right).
}
To estimate the integral along $\Rep (s) = -1$, we take the logarithmic derivative of both sides of the functional equation \eqref{eqn:fncL} and write
\begin{align*}
    \frac{1}{2}\log\left( \frac{q}{4\pi^2}\right)  +\frac{\Gamma'}{\Gamma}\left(s+\frac{k}{2}\right)&+\frac{L'}{L}\left(s + \frac 12,f\right)\\
    &= -\frac{1}{2} \log\left( \frac{q}{4\pi^2}\right) - \frac{\Gamma'}{\Gamma}\left(-s+\frac{k}{2}\right) - \frac{L'}{L}\left(\frac 12 - s,\bar f\right). 
\end{align*}
Hence, if we denote $G(s) = \frac{\Gamma'}{\Gamma}\left(s+\frac{k}{2}\right) + \frac{\Gamma'}{\Gamma}\left(-s+\frac{k}{2}\right)$, then
\begin{align*}
  & \frac{1}{2\pi i }\int_{(-1)}\frac{L'}{L}\left(s + \frac 12,f\right)\Phi\left(-is \frac{\log q}{2\pi}\right)\,ds\\
    &= -\frac{1}{2\pi i }\int_{(0)} \left[\log\left( \frac{q}{4\pi^2}\right)  + G(s) \right]\Phi\left(-is \frac{\log q}{2\pi}\right)\,ds -  \frac{1}{2\pi i }\int_{(1)}\frac{L'}{L}\left(s + \frac 12,\bar f\right) \Phi\left(-is \frac{\log q}{2\pi}\right) \> ds \\
    &= - \int_{-\infty}^{\infty} \Phi(x) \> dx + O_k \left( \frac{1}{\log q}\right) -  \frac{1}{2\pi i }\int_{(1)}\frac{L'}{L}\left(s + \frac 12,\bar f\right) \Phi\left(-is \frac{\log q}{2\pi}\right)\> ds.
\end{align*}
Finally, we use the series representation for $ \frac{L'}{L}(s,f)$ for $\Rep(s) > 1$ to finish the proof.  To be precise, for $\Rep(s) > 1$, 
$$-\frac{L'}{L}(s,f)=\sum_{n=1}^{\infty}\frac{\Lambda(n)c_f(n)}{n^s}. $$
Therefore
\begin{align*}
    \frac{1}{2\pi i}\int_{(1)}\frac{L'}{L}\left(s + \frac 12, f\right)\widehat \Phi\left(-is \frac{\log q}{2\pi}\right)\,ds
    & = -\sum_{n=1}^\infty \frac{\Lambda(n) c_f(n)}{\sqrt n} \frac{1}{2\pi i }\int_{(1)}\Phi\left(-is \frac{\log q}{2\pi}\right) \frac{1}{n^s}\, ds\\
    & = - \frac{1}{\log q}\sum_{n=1}^\infty \frac{\Lambda(n) c_f(n)}{\sqrt n}\widehat\Phi\left(\frac{\log n}{\log q}\right).
\end{align*}
\end{proof}

In order to bound sums involving primes, we will use the following standard Lemma.
\begin{lem} \label{lem:boundforL'overL} Assume GRH for $L(s, f)$, where $f$ is a primitive holomorphic Hecke eigenform or a primitive Maass Hecke eigenform of level $q$ and weight $k$. For $s = \sigma + i\tau$ with $\frac 12 < \sigma \leq \frac 54$, we have
 \begin{equation*}
-\frac{L'}{L}(s, f) \ll \frac{\big( \log(q + k + |\tau|)\big)^{^{\frac 43 - \frac{2\sigma}{3}} }}{2\sigma - 1}.
\end{equation*}
\end{lem}
\begin{proof}
We will modify the proof of Theorem~5.17 in \cite{IK} without assuming the Ramanujan-Petersson conjecture for $L(s, f)$. Equation (5.59) in \cite[Chapter 5]{IK} implies that if $Z \geq 1$ and $s = \sigma + i\tau$ with $\frac 12 < \sigma \leq \frac 54$, then
\begin{equation} \label{eqn:L'LfromIK} - \frac{L'}{L}(f, s) = \sum_{n} \frac{\Lambda_f(n)}{n^s} \phi\left( \frac nZ\right) + O\left( \frac{\log (q + k + |\tau|)}{2\sigma - 1}Z^{\frac 12 - \sigma}\right).
\end{equation}
Here, $\phi(y)$ is a non-negative continuous function on $[0, \infty)$ whose Mellin transform $\tilde \phi(w)$ satisfies
$$ w(w+1) \tilde \phi(w) \ll 1 \ \ \ \ \textrm{for $w$ such that} \ \ \ - \frac 12 \leq \tRe(w) \leq  2.$$
Moreover, $\tilde\phi(w)$ has a simple pole at $w = 0$ with residue $1$.  For instance, we may take $\phi(y) = e^{-y}$ and $\tilde\phi(w) = \Gamma(w)$.

We know that $|\Lambda_f(n)| \leq n$ and it follows that
$$ \sum_n \frac{\Lambda_f(n)}{n^s} \phi\left( \frac n Z\right) \ll \sum_n n^{1-\sigma} \phi\left( \frac n Z\right)  \ll Z^{2-\sigma},$$ via a standard contour integration argument.  
Choosing $Z = \log^{2/3}(q + k + |\tau|) $ in \eqref{eqn:L'LfromIK} and the bound above, we obtain our desired result.
\end{proof}

{\lem\label{lem:CLee3.5} Assume GRH for $L(s,\chi)$ with $\chi$ mod $q$ and for $L(s, f)$, where $f$ is a primitive holomorphic Hecke eigenform or a primitive Maass Hecke eigenform of level $q$ and weight $k$. Let $X>0$ be a real number, and let $\Psi$ be a smooth function that is compactly supported on $[0,X]$. Suppose that, for each positive integer $m$, there exists a constant $A_m$ depending only on $m$ such that
$$
|\Psi^{(m)}(x)| \leq \frac{A_m}{\min\{\log(X+3), X/x\}x^{m}}
$$
for all $x>0$. Write $z = \frac 12 +it$ with $t$ real, and let $N$ be a positive integer. If $\chi$ is a non-principal character, then
$$\sum_{(p, N) = 1}\frac{\chi(p)\log(p)\Psi(p)}{p^z}\ll A_3 \log^{1 + \epsilon}( X + 2) \log(q+|t| ) + \log N \max_{0 \leq x \leq X} |\Psi(x)|,$$
with absolute implied constant. Similarly, 
$$\sum_{(p, N) = 1}\frac{\lambda_f(p)\log(p)\Psi(p)}{p^z}\ll A_3 \log^{1 + \epsilon}( X + 2) \log (q+k+|t| ) + \log N \max_{0 \leq x \leq X} |\Psi(x)|,$$
with absolute implied constant.}

\noindent \textit{Remark}: If $c$ is a fixed constant and $\Upsilon$ is a smooth function compactly supported on $[0,c]$, then the function $\Psi(x)=\Upsilon(cx/X)$ satisfies the conditions in Lemma~\ref{lem:CLee3.5} since $X^{-m} \ll x^{-m} (x/X)$ for positive integers $m$. Also, if $\Upsilon$ is a smooth function compactly supported on $(-\infty,c]$, then the function $\Psi(x)=\Upsilon(\frac{c\log x}{\log X})$ satisfies the conditions in the lemma.

\begin{proof}
Observe first that, by repeated integration by parts, the Mellin transform $\widetilde{\Psi}$ of $\Psi$ satisfies
\begin{align*}
\widetilde{\Psi}(s) 
& = -\frac{1}{s(s+1)(s+2)} \int_0^X  \Psi^{(3)}(x) x^{s+2}\,dx \\
& \ll \frac{A_3}{|s(s+1)(s+2)|}\Bigg(\int_0^{X/\log (X+3)} \frac{x^{\text{Re}(s)-1}}{\log(X+3)} \,dx + \int_{X/\log (X+3)}^{X} \frac{x^{\text{Re}(s)}}{X}\Bigg)\\
& \ll \frac{A_3}{|s(s+1)(s+2)|} \Bigg(\frac{X^{\text{Re}(s)}}{\text{Re}(s)\log(X+3)} + \frac{X^{\text{Re(s)}}}{\text{Re}(s)+1}\Bigg)
\end{align*}
for Re$(s)>0$. Thus, if Re$(s)= 1/\log (X+2)$, then
\begin{equation}\label{modifiedCLee3.5mellin}
\widetilde{\Psi}(s) \ll \frac{A_3}{|s(s+1)(s+2)|}.
\end{equation}

Now to prove the first assertion of the lemma,  we write 
\est{ \sum_{(p, N) = 1}\frac{\chi(p)\log(p)\Psi(p)}{p^z} &= \sum_{p}\frac{\chi(p)\log(p)\Psi(p)}{p^z} - \sum_{p | N}\frac{\chi(p)\log(p)\Psi(p)}{p^z}.}
The second sum on the right-hand side is bounded by $\sum_{p | N} \log p \ll \log N$ times the maximum value of $|\Psi|$. We may bound the first sum by slightly modifying the proof of Lemma~2.3 in \cite{CKLL} and using the Mellin transform bound \eqref{modifiedCLee3.5mellin}. The difference is that here we move the contour to $\frac 1{\log (X+2)}$ instead of $\frac{20}{\log q}$, and use the estimate 
$$ \frac{L'}{L}(s, \chi) \ll \log( X + 2) \log (q +  |t| + 2),$$
which follows from Theorem~5.17 of \cite{IK} when $\sigma =\frac 12 +  \frac 1{\log (X+2)}$.  The extra factor of $\log^\epsilon (X+2)$ arises from \eqref{modifiedCLee3.5mellin} and
$$\int_0^\infty \left(\frac{1}{\frac{1}{\log (X+2)} + t}\right) \left(\frac{\log(t+2)}{(t+1)(t+2)}\right) dt \ll \int_0^1 \frac{1}{\frac{1}{\log (X+2)} + t} dt + 1 \ll \log \log (X+3)
$$via a change of variables.

The proof in the case of $L(s, f)$ is similar.  We first deal with the coprimality condition by writing 
\begin{equation}\label{eqn: CLee3.5Maass}
\sum_{(p, N) = 1}\frac{\lambda_f(p)\log(p)\Psi(p)}{p^z} = \sum_{p}\frac{\lambda_f(p)\log(p)\Psi(p)}{p^z} - \sum_{p |N}\frac{\lambda_f(p)\log(p)\Psi(p)}{p^z}.
\end{equation}
We may bound the first sum on the right-hand side in a way similar to the case with Dirichlet characters. Indeed, the main input is the estimate
\begin{equation*}
\frac{L'}{L}(s, f) \ll \log(q+k+|\tau|)\log (X+2),
\end{equation*}
which holds for $s = \sigma_0 + i\tau$ with $\sigma_0 = \frac 12 + \frac{1}{\log (X+2)}$ and is a consequence of Lemma~\ref{lem:boundforL'overL}. Thus, the first sum on the right-hand side of \eqref{eqn: CLee3.5Maass} is
$$
\ll A_3\log^{1 + \epsilon} (X + 2)\log (q+k+|t| ).
$$
As for the second sum, we may use the crude bound $\lambda_f(p) \ll \sqrt p$ to deduce that
$$ \sum_{p |N}\frac{\lambda_f(p)\log(p)\Psi(p)}{p^z} \ll \sum_{p | N} \log p \max_{0 \leq x \leq X} |\Psi(x)| \ll \log N\max_{0 \leq x \leq X} |\Psi(x)|. $$
\end{proof}

\section{Kuznetsov's formula} \label{sec:kutz}

In this section, we will state some relevant results from spectral theory. We refer the reader to \cite{DI} and \cite{Iwaniec} for background reading.

We now introduce some notation that will appear in Kuznetsov's formula. There are three parts in Kuznetsov's formula---contributions from holomorphic forms, Maass forms, and Eisenstein series---and we now define the Fourier coefficients of these forms.

\subsubsection*{Holomorphic forms}

Let $z = x + iy$. Let $B_{\ell}(q)$ be an orthonormal basis of the space of holomorphic cusp forms of weight $\ell$ level $q$, and  $\theta_{\ell}(q)$ is the dimension of the space $S_{\ell}(q)$. We can write $B_{\ell}(q) = \{f_1, f_2,...., f_{\theta_{\ell}(q)} \}$, and the Fourier expansion of $f_j \in B_{\ell}(q)$ can be expressed as follows
$$ f_j(z) = \sum_{n \geq 1} \psi_{j, \ell}(n) (4\pi n)^{\ell / 2} \e{n z}.$$

We call $f$ a Hecke eigenform if it is an eigenfunction of all the Hecke operators $T(n)$ for $(n, q) = 1$.  In that case, we denote the Hecke eigenvalue of $f$ for $T(n)$ as $\lambda_f(n)$.  Writing $\psi_f(n)$ as the Fourier coefficient, we have that
$$\lambda_f(n) \psi_f(1) = \sqrt{n} \psi_f(n),
$$for $(n, q) = 1$.  When $f$ is a newform, this holds in general.  We also have the Ramanujan bound
$$\lambda_f(n) \ll \tau(n) \ll n^\epsilon.
$$

\subsubsection*{Maass forms}
Let 
$
\lambda_j := \frac{1}{4}+\kappa_j^2,
$
where
$
0=\lambda_0 \leq \lambda_1\leq \lambda_2 \leq \dots
$
are the eigenvalues, each repeated according to multiplicity, of the Laplacian $-y^2 ( \frac{\partial^2}{\partial x^2} + \frac{\partial^2}{\partial y^2})$ acting as a linear operator on the space of cusp forms in $L^2(\Gamma_0(q) \backslash \mathbb{H})$, where by convention we choose the sign of $\kappa_j$ that makes $\kappa_j\geq 0$ if $\lambda_j\geq \frac{1}{4}$ and $i\kappa_j >0$ if $\lambda_j <\frac{1}{4}$. For each of the positive $\lambda_j$, we may choose an eigenvector $u_j$ in such a way that the set $\{u_1,u_2,\dots\}$ forms an orthonormal system, and we define $\rho_j(m)$ to be the $m$th Fourier coefficient of $u_j$, i.e.,
$$
u_j(z) =  \sum_{m\neq 0} \rho_j(m)W_{0, i\kappa_j} (4\pi |m|y) \e{mx}
$$
with $z=x+iy$, where $W_{0, it}(y) =  \left( y/\pi\right)^{1/2}K_{it}(y/2)$ is a Whittaker function, and $K_{it}$ is the modified Bessel function of the second kind. 

We call $u$ a Hecke eigenform if it is an eigenfunction of all the Hecke operators $T(n)$ for $(n, q) = 1$.  In that case, we denote the Hecke eigenvalue of $u$ for $T(n)$ as $\lambda_u(n)$.  Writing $\rho_u(n)$ as the Fourier coefficient, we have that
\begin{equation}\label{eqn: fouriercoeffintermsofeigenvalue}
\lambda_u(n) \rho_u(1) = \sqrt{n} \rho_u(n)
\end{equation}
for $(n, q) = 1$.  When $u$ is a newform, this holds in general.  We also have that
$$\lambda_u(n) \ll \tau(n) n^{\theta} \ll n^{\theta+ \epsilon},
$$where we may take $\theta = \frac{7}{64}$ due to work of Kim and Sarnak \cite{KimS}.

\subsubsection*{Eisenstein series}
Let $\mathfrak{c}$ be a cusp for $\Gamma_0(q)$.  We define $\varphi_{\mathfrak{c}}(m, t)$ to be the $m$th Fourier coefficient of the (real-analytic) Eisenstein series at $1/2 + it$:
\begin{align*}
E_{\mathfrak{c}}(z; 1/2 + it)  &=
  \delta_{\mathfrak{c} = \infty} y^{1/2 + it} + \varphi_{\mathfrak{c}}(0,t) y^{1/2 - it}  + \sum_{m\neq 0}  \varphi_{\mathfrak{c}} (m, t) W_{0, it} (4\pi |m|y) \e{mx},
\end{align*}
where $z=x+iy$.

\subsubsection*{Kuznetsov's formula} We state the version given by Lemma~10 of \cite{BM}.
\begin{lem}\label{lem:kuznetsov}
Let $\phi:(0,\infty)\rightarrow \mathbb{C}$ be smooth and compactly supported, and let $m,n,q$ be positive integers. Then
\begin{align*}
\sum_{\substack{c\geq 1 \\ c\equiv 0 \bmod{q}}} \frac{S(m,n;c)}{c} \phi
& \bigg( 4\pi \frac{\sqrt{mn}}{c}\bigg) = \sum_{j=1}^{\infty} \frac{\overline{\rho_j}(m)\rho_j(n) \sqrt{mn}}{\cosh(\pi \kappa_j)}\phi_+(\kappa_j)\\
& + \frac{1}{4\pi} \sum_{\mathfrak{c}} \int_{-\infty}^{\infty} \frac{ \sqrt{mn}}{\cosh (\pi t)} \, \overline{\varphi_{\mathfrak{c}} (m, t)}  \varphi_{\mathfrak{c} }(n, t)  \phi_+ (t) \,dt \\
& + \sum_{ \substack{\ell \geq 2 \mbox{\scriptsize{ \upshape{even}}} \\ 1 \leq j \leq \theta_{\ell}(q)} } (\ell-1)! \sqrt{mn} \, \overline{\psi_{j,\ell}}(m) \psi_{j,\ell} (n) \phi_h(\ell),
\end{align*}
where the Bessel transforms $\phi_+$ and $\phi_h$ are defined by
$$
\phi_+(r):=\frac{2\pi i}{\sinh(\pi r)} \int_0^{\infty} (J_{2ir}(\xi) - J_{-2ir} (\xi) ) \phi(\xi) \,\frac{d\xi}{\xi}
$$
and
$$
\phi_h(\ell) := 4 i^k \int_0^{\infty} J_{\ell - 1}(\xi) \phi(\xi) \,\frac{d\xi}{\xi},
$$
where $J_{\ell - 1}$ is the Bessel function of the first kind.

\end{lem}

We next state bounds for the transforms $\phi_+$ and $\phi_h$ that appear in Kutznetsov's formula. These bounds are consequences of Lemma~1 of \cite{BHM}. 
\begin{lem} \label{lem:boundforcoefficientafterKuznetsov}
\begin{enumerate}
\item Let $\phi(x)$ be a smooth function supported on $x \asymp X$ such that $\phi^{(j)}(x) \ll_{j} X^{-j}$ for all integers $j \geq 0$. For $t \in \mathbb R$, we have
$$ \phi_+(t), \ \  \phi_h(t) \ll_C \frac{1 + |\log X|}{1 + X} \left( \frac{1 + X}{1 + |t|}\right)^C$$
for any constant $C \geq 0.$
\item  Let $\phi(x)$ be a smooth function supported on $x \asymp X$ such that $\phi^{(j)}(x) \ll_{j} (X/Z)^{-j}$ for all integers $j \geq 0$. For $t \in (-1/4, 1/4)$, we have
$$ \phi_+ (it) \ll \frac{1 + (X/Z)^{-2 |t|}}{1 + X/Z}.$$
\item Assume that $\phi(x) = e^{iax} \psi(x)$ for some constant $a$ and some smooth function $\psi(x)$ supported on $x \asymp X$ such that $\psi^{(j)}(x) \ll_j X^{-j}$ for all integers $j \geq 0.$  Then
$$ \phi_+(t), \ \  \phi_h(t) \ll_{C, \epsilon} \frac{1+|\log X|}{F^{1 -\epsilon}} \left( \frac{F}{1 + |t|} \right)^C$$
for any $C \geq 0,$ $\epsilon > 0$ and some $F = F(X, a) < (|a|+1)(X+1).$
\end{enumerate} 
\end{lem}
Our Lemma~\ref{lem:boundforcoefficientafterKuznetsov}~(3) is a slight generalization of Lemma~1~(c) in \cite{BHM}. This generalization incorporates the bound in Lemma~1~(a) in \cite{BHM}. It is convenient for us that Lemma~\ref{lem:boundforcoefficientafterKuznetsov}~(3) holds uniformly for all $a$.

Next, we record the following bounds, which we deduce from our previous lemmas.
\begin{lem} \label{lem:boundforhu}
Suppose that $W$ is a smooth function that is compactly supported on $(0,\infty)$. For real $X>0$ and real numbers $u$ and $\xi$, let 
\begin{align*}
h_u(\xi)  = J_{k-1}(\xi)W\bfrac{\xi}{X} \e{u\xi}.
\end{align*}
Then 
\begin{enumerate}
    \item $h_+(r) \ll    \frac{1 + |\log X|}{F^{1 - \epsilon}} \left( \frac{F}{1 + |r|}\right)^C \min \left\{ X^{k - 1}, \frac 1{\sqrt X}\right\} \ \ \ \ \ \ \ \   \textrm{for some} \ F < (|u| + 1)(1 + X). $
    \item If $r\in (- 1/4,1/4)$, then $h_+(ir) \ll  \left( \frac 1{\sqrt X} + (1 + |u|)^{\frac 12}\right) \min \left\{ X^{k - 1}, \frac 1{\sqrt X}\right\}. $ 
    
\end{enumerate}

\end{lem}
\begin{proof}
Since $W$ is compactly supported, we may assume that $\xi \asymp X$. For $ X \ll 1$, we use the fact that
\begin{equation}\label{eqn:Jtaylor}
 J_{k-1}(\xi) = \sum_{\ell = 0} (-1)^{\ell} \frac{(\xi/2)^{2\ell + k-1}}{\ell ! (\ell + k - 1)!}
\end{equation}
by Lemma~\ref{jbessel}. Note that the first term dominates for $\xi \ll 1$. On the other hand, for $X \gg 1$, we use the fact that
\es{\label{eqn:besselbigx} J_{k-1}(\xi) =\frac{1}{ \sqrt{2\pi \xi}}\left(W_k(\xi) \e{\frac{\xi}{2\pi}-\frac{k}{4}+\frac{1}{8}} + \overline{W}_k(\xi) \e{-\frac{\xi}{2\pi}+\frac{k}{4}-\frac{1}{8} }\right),} where $W_k^{(j)}(x)\ll_{j,k} x^{-j}.$
    
Let 
\begin{align*}
g(\xi) : = \max\{ X^{1-k}, \sqrt X\} h_u(\xi), 
\end{align*}
so that
\begin{equation*}
h_u(\xi) = \min\left\{ X^{k-1}, \frac 1{\sqrt X} \right\} g(\xi)
\end{equation*}and 
\begin{equation*}
h_+(r) =  \min\left\{ X^{k-1}, \frac 1{\sqrt X} \right\} g_+(r).
\end{equation*}

In applying Lemma~\ref{lem:boundforcoefficientafterKuznetsov} below, we use the following bounds. For $X \ll 1$, we use \eqref{eqn:Jtaylor} and the product rule to deduce that
\begin{align*}
\frac{d^{j}}{d\xi^j}X^{1 - k} J_{k-1}(\xi)W\left( \frac{\xi}{X}\right) \ll_j X^{-j} \ \ \text{ and}\\
\frac{d^{j}}{d\xi^j}g(\xi) \ll_j  |u|^{j} + X^{-j}
\end{align*}
for all integers $j \geq 0$. For $X \gg 1$, we use \eqref{eqn:besselbigx} and the product rule to deduce that
\begin{align*}
\frac{d^{j}}{d\xi^j} \sqrt{\frac{X}{\xi}} W_k(\xi) W\left( \frac{\xi}{X}\right) \ll_j X^{-j} \ \text{ and}\\
\frac{d^{j}}{d\xi^j}g(\xi) \ll_j  \left(\frac {1}{2\pi} + |u|\right) ^{j} + X^{-j} \ll (1 + |u|)^j
\end{align*}for all integers $j \geq 0$, since $\frac 1X \ll 1 \ll \frac{1}{2\pi} + |u|$.

From these and Lemma~\ref{lem:boundforcoefficientafterKuznetsov}~(c), we arrive at
\begin{align*} 
g_+(r) &\ll_{C, \epsilon} \frac{1+|\log X|}{F^{1 -\epsilon}} \left( \frac{F}{1 + |r|} \right)^C  &\textrm{for some} \ F < (|u| + 1)(X + 1).
\end{align*}
This proves Lemma~\ref{lem:boundforhu}~(1).

We next bound $g_+(ir)$ for $r\in (-1/4,1/4)$. For $X \ll 1$, if $|u| \ll \frac 1X$, then $\frac{d^{j}}{d\xi^j}g(\xi) \ll_j   X^{-j}$ and thus Lemma~\ref{lem:boundforcoefficientafterKuznetsov}~(b) implies 
\begin{align*}
 g_+(ir) &\ll \frac{1 + X^{-2|r|}}{1 + X} \ll \frac{1}{\sqrt X} &\textrm{  for } r\in (-1/4,1/4). 
 \end{align*}
 On the other hand, if $|u| \gg \frac 1X$, then
 $\frac{d^{j}}{d\xi^j}g(\xi) \ll_j   \left( \frac{X}{X|u|}\right) ^{-j}$, and so
\begin{align*}
g_+(ir) &\ll  \frac{1 + |u|^{2 |r|}}{1 + \frac{1}{|u|}} \ll (1 + |u|)^{1/2} &\textrm{  for } r\in (-1/4,1/4). 
\end{align*}
For $X \gg 1,$ Lemma~\ref{lem:boundforcoefficientafterKuznetsov}~(b) implies
\begin{align*}
g_+(ir) &\ll  \frac{1 + (1 + |u|)^{2 |r|}}{1 + \frac{1}{ 1 + |u|}} \ll (1 + |u|)^{1/2} &\textrm{  for } r\in (-1/4,1/4). 
\end{align*}
In summary, we have for $r\in (-1/4,1/4)$ that
\begin{align*}
g_+(ir) \ll \frac{1}{\sqrt X} + (1 + |u|)^{\frac 12}. 
\end{align*}
This proves the lemma.

\end{proof}

\subsection{Oldforms and newforms}\label{subsec:newformoldform}
When using Petersson's formula or Kutznetsov's formula, we will frequently come across an orthonormal basis of Maass forms $\{u_j\}$, or an orthonormal basis of holomorphic modular forms $B_\ell(q)$.  However, in order to apply GRH for Hecke $L$-functions to bound our sum over primes, we need to express our basis in terms of newforms.  This theory is due to Atkin and Lehner~\cite{AL}.  We refer the reader to \S 14.7 of the book \cite{IK}, \S 2 of \cite{ILS} and \S 5 of \cite{BM} for more background.

We will state this theory for Maass forms only, although the theory applies also to holomorphic modular forms with slight changes in notation.  Let $S(\N)$ denote the space of all Maass forms of level $\N$ and $S^*(\M)$ denote the space that is orthogonal to all old forms of level $\M$.  By work of Aktin and Lehner~\cite{AL}, $S^*(\M)$ has an orthonormal basis consisting of primitive Hecke eigenforms, which we denote by $H^{*}(\M)$.  Then, we have the orthogonal decomposition
$$
S(\N) = \bigoplus_{\N = \mL \M} \bigoplus_{f \in H^*(\M)} S(\mL; f),
$$
where $S(\mL; f)$ is the space spanned by $f|_l$ for $l|\mL$, where 
$$
f|_l(z) = f(lz).
$$

Let $f$ denote a newform of level $\M|\N$, normalized as a level $\N$ form, which means that the first coefficient satisfies 
\begin{equation*}
|\rho_{ f}(1)|^2 = \frac{(\N t_f)^{o(1)}}{\N},
\end{equation*}
where $t_f$ is the spectral parameter of $f$.

Blomer and Milicevic showed in Lemma 9 of \S 5 of \cite{BM} that there is an orthonormal basis for $S(\mL; f)$ of the form $f^{(g)}$ for $g|\mL$, where
\begin{equation}\label{eqn:BMorthonormalform}
f^{(g)} = \sum_{d|g} \xi_g(d) f|_d,
\end{equation}
where $\xi_g(d)$ is defined in (5.6) of \cite{BM} and satisfies
\begin{equation*}
\xi_g(d) \ll g^\epsilon \bfrac{g}{d}^{\theta - 1/2}.
\end{equation*}
For convenient reference, we will call this basis the Hecke basis of level $\N$.

Lemma 9 of \cite{BM} implies that the Fourier coefficients of $f^{(g)}$ satisfy
\begin{equation} \label{eqn:fouriercoeffbdd}
\sqrt{n} \rho_{f^{(g)}}(n) \ll (n\N)^\epsilon n^\theta (\N, n)^{1/2 - \theta} |\rho_{f}(1)| \ll \N^\epsilon n^{1/2+\epsilon} |\rho_{f}(1)|.
\end{equation}
This bound is somewhat crude, but will suffice for our purposes. Note that $f^{(g)}$ is an eigenfunction of the Hecke operator $T(n)$ for all $(n, \mathcal N)  =1$. Indeed, the $n$th Fourier coefficient of $f|_d$ is nonzero only if $d|n$. Since $g|\N$ in \eqref{eqn:BMorthonormalform}, it follows for $(n, \N) = 1$ that we may take only the $d=1$ term and deduce that
\begin{align*}
\rho_{f^{(g)}}(n) = \xi_g(1) \rho_f(n) = \xi_g(1) \rho_f(1) \lambda_f(n) = \rho_{f^{(g)}}(1) \lambda_f(n).
\end{align*}
This implies that $f^{(g)}$ is a Hecke eigenform with
\begin{equation}\label{eqn: eigenvaluesoff(g)intermsoff}
\lambda_{f^{(g)}}(n) = \lambda_f(n)
\end{equation}
for $(n, \N) = 1$.

\subsubsection*{Remark}
For the rest of the paper, we will always take our orthonormal basis of cusp Maass forms $\{u_j\}$ and orthonormal basis of holomorphic forms $B_l(\N)$ to be these Hecke bases defined above.


\begin{lem}\label{lem:heckebasisfouriersumbdd}
Let $u$ be an element of the Atkin-Lehner basis of level $\N$ so $u = f^{(g)}$ for some primitive Hecke form $f$ of level $\M|\N$, and some $g|\frac{\N}{\M}$.  Let $\Psi$ be a smooth function supported in $(a,b)$, where $0 < a < b.$  Write $z = \frac 12 +it$ for real $t$.  Finally, let $X>0$ be a real number. Then
$$\sum_{p}\frac{\sqrt{p} \rho_u(p)\log(p)\Psi(p/X)}{p^z} \ll_{\Psi} |\rho_u(1)|  \log(\N+|t| ) \log^{1 + \epsilon}(X+2)  +|\rho_f(1)|\N^\epsilon.$$
\end{lem}
\begin{proof}
We have that
\begin{align*}
\sum_{p}\frac{\sqrt{p} \rho_u(p)\log(p)\Psi(p/X)}{p^z}
&= \sum_{(p, \N) = 1}\frac{\sqrt{p} \rho_u(p)\log(p)\Psi(p/X)}{p^z} + O \left( \sum_{p|\N} \N^\epsilon p^{\epsilon} |\rho_{f}(1)| \right)
\end{align*}by \eqref{eqn:fouriercoeffbdd}. The second quantity on the right-hand side is visibly $\ll |\rho_f(1)| \N^\epsilon$, while the first equals
\begin{align*}
\rho_u(1) \sum_{(p, \N) = 1}\frac{\lambda_f(p)\log(p)\Psi(p/X)}{p^z}
&\ll |\rho_u(1)| \log(\N+|t| ) \log^{1 + \epsilon} (X+2)
\end{align*}
by \eqref{eqn: fouriercoeffintermsofeigenvalue}, \eqref{eqn: eigenvaluesoff(g)intermsoff}, and Lemma~\ref{lem:CLee3.5}. We may apply Lemma~\ref{lem:CLee3.5} because
\begin{equation*}
\frac{d^m}{dx^m} \Psi\left( \frac{x}{X}\right) = \frac{1}{X^m}\Psi^{(m)} \left( \frac{x}{X}\right) \ll \frac{1}{Xx^{m-1}} \max_{a<y<b} |\Psi^{(m)}(y)|
\end{equation*}
for each positive integer $m$.
\end{proof}



\section{Initial Setup for the Proof of Proposition \ref{prop:main}} \label{sec:mainsetup}

Our first task is to trim the sum over prime powers in Lemma~\ref{lem:Explicit} and prove the following.

\begin{lem}\label{lem:trimpsum}
Let the notations be as in Lemma~\ref{lem:Explicit}. We have
\begin{align*}
\frac{1}{\log q}\sum_n \frac{\Lambda(n) c_f(n)}{\sqrt{n}} \widehat \Phi\left(\frac{\log n}{\log q} \right) = \frac{1}{\log q}\sum_{p\nmid q} \frac{\lambda_f(p) \log p}{\sqrt{p}} \widehat \Phi\left(\frac{\log p}{\log q} \right)\\
+\frac{1}{\log q}\sum_{p\nmid q} \frac{\lambda_f(p^2)\log p}{p} \widehat \Phi\left(\frac{\log p^2}{\log q} \right) 
- \frac{1}{4} \Phi (0) +O\left(\frac{\log \log (3q)}{\log q}\right).
\end{align*}

\end{lem}

\begin{proof}
We first split the sum into three parts and write
\begin{align*}
& \frac{1}{\log q}\sum_n \frac{\Lambda(n) c_f(n)}{\sqrt{n}} \widehat \Phi\left(\frac{\log n}{\log q} \right) \\
& = \frac{1}{\log q} \sum_p \frac{c_f(p)\log p}{\sqrt{p}} \widehat \Phi\left(\frac{\log p}{\log q} \right) 
+\frac{1}{\log q} \sum_{p} \frac{c_f(p^2)\log p}{p}\widehat \Phi\left(\frac{\log p^2}{\log q} \right)\\
& \hskip 2in  + \frac{1}{\log q}\sum_p \sum_{\substack{ \nu \geq 3}} \frac{c_f(p^{\nu})\log p}{p^{\nu/2}}\widehat \Phi\left(\frac{\log p^{\nu}}{\log q} \right).
\end{align*}
Since $|c_f(p^{\nu})|\leq 2$ for all $\nu\geq 0$ (see \eqref{def:cfn} and the proof of Lemma~4.1 in \cite{ILS}) and $|\widehat{\Phi}(x)|\leq \int_{-\infty}^{\infty} |\Phi(x)|\,dx$, we have
\begin{align*}
\frac{1}{\log q}\sum_p \sum_{\substack{ \nu \geq 3}} \frac{c_f(p^{\nu})\log p}{p^{\nu/2}}\widehat \Phi\left(\frac{\log p^{\nu}}{\log q} \right) \ll \frac{1}{\log q}.
\end{align*}
Also, \eqref{def:Lsf} and \eqref{def:cfn} imply
\begin{align*}
c_f(p) = \alpha_f(p)+\beta_f(p) = \lambda_f(p),
\end{align*}
while the multiplicativity of $\lambda_f(n)$ (see (2.19) of \cite{ILS}) and \eqref{def:Lsf} imply
\begin{align*}
\lambda_f(p^2) = \alpha_f^2(p)+\alpha_f(p)\beta_f(p)+\beta_f^2(p)= \alpha_f^2(p)+\beta_f^2(p)+\chi_0(p),
\end{align*}
and thus it follows from \eqref{def:cfn} that
\begin{align*}
c_f(p^2) = \alpha_f^2(p)+\beta_f^2(p) = \lambda_f(p^2) - \chi_0(p).
\end{align*}
Therefore
\es{ \label{eqn:afterboundprimepowermorethan3}
& \frac{1}{\log q}\sum_n \frac{\Lambda(n) c_f(n)}{\sqrt{n}} \widehat \Phi\left(\frac{\log n}{\log q} \right) \\
& = \frac{1}{\log q}\sum_p \frac{\lambda_f(p) \log p}{\sqrt{p}} \widehat \Phi\left(\frac{\log p}{\log q} \right) 
+\frac{1}{\log q}\sum_{p} \frac{(\lambda_f(p^2)-\chi_0(p))\log p}{p} \widehat \Phi\left(\frac{\log p^2}{\log q} \right) 
+O\left(\frac{1}{\log q}\right).
}

We next bound the contribution of the terms with $p|q$. If $p|q$, then $\lambda_f^2(p) =p^{-1}$ (see the proof of Lemma~4.1 in \cite{ILS}). Hence
\es{\label{eqn1:p|q}
\frac{1}{\log q}\sum_{p|q} \frac{\lambda_f(p) \log p}{\sqrt{p}} \widehat \Phi\left(\frac{\log p}{\log q} \right) \ll \frac{1}{\log q}\sum_{p|q} \frac{ \log p}{p} \ll \frac{\log \log (3q)}{\log q}
}
and
\es{\label{eqn2:p|q}
\frac{1}{\log q}\sum_{p|q} \frac{(\lambda_f(p^2)-\chi_0(p))\log p}{p} \widehat \Phi\left(\frac{\log p^2}{\log q} \right) = \frac{1}{\log q}\sum_{p|q} \frac{\log p}{p^2} \widehat \Phi\left(\frac{\log p^2}{\log q} \right) \ll \frac{1}{\log q}
}
since if $p|q$ then $\chi_0(p)=0$ and $\lambda_f(p^2)=\lambda_f^2(p)$ by (2.19) of \cite{ILS} (the Hecke relations).

We next simplify the total contribution of $\chi_0(p)$. It is
\es{ \label{eqn:trivchi}
-\frac{1}{\log q}\sum_{p\nmid q} \frac{\chi_0(p)\log p}{p} \widehat \Phi\left(\frac{\log p^2}{\log q} \right) 
& = -\frac{1}{\log q}\sum_{p} \frac{\log p}{p} \widehat \Phi\left(\frac{\log p^2}{\log q} \right) + \frac{1}{\log q}\sum_{p| q} \frac{\log p}{p} \widehat \Phi\left(\frac{\log p^2}{\log q} \right) \\
& = -\frac{1}{\log q}\sum_{p} \frac{\log p}{p} \widehat \Phi\left(\frac{\log p^2}{\log q} \right) + O\left(\frac{\log \log (3q)}{\log q}\right).
}
To estimate this, we use partial summation and the prime number theorem to deduce that
\es{\label{eqn:trivchiwithpnt} -\frac{1}{\log q}\sum_{p} \frac{\log p}{p} \widehat \Phi\left(\frac{\log p^2}{\log q} \right) =  - \frac{1}{4}{\Phi}(0)  + O\left(\frac{\log \log (3q)}{\log q}\right).}

Combining \eqref{eqn:afterboundprimepowermorethan3} -- \eqref{eqn:trivchiwithpnt}, we obtain the desired result.
\end{proof}

We average both sides of the conclusion of Lemma~\ref{lem:trimpsum} with respect to $f$ and then with respect to $q$ to deduce that
\es{ \label{eqn:afterapplyingexplicitformula}
\frac{1}{N(Q)}
& \sum_q \Psi\bfrac{q}{Q} \frac{1}{\log q} \sumh_{f \in \mathcal H_{k}(q)}\sum_n \frac{\Lambda(n) c_f(n)}{\sqrt{n}} \widehat \Phi\left(\frac{\log n}{\log q} \right) \\
= & \frac{1}{N(Q)}\sum_q \Psi\bfrac{q}{Q} \frac{1}{\log q}\sumh_{f \in \mathcal H_{k}(q)}\sum_{p\nmid q} \frac{\lambda_f(p) \log p}{\sqrt{p}} \widehat \Phi\left(\frac{\log p}{\log q} \right)  \\
& + \frac{1}{N(Q)}\sum_q \Psi\bfrac{q}{Q}\frac{1}{\log q} \sumh_{f \in \mathcal H_{k}(q)}\sum_{p\nmid q} \frac{\lambda_f(p^2)\log p}{p} \widehat \Phi\left(\frac{\log p^2}{\log q} \right)  \\
& - \frac{1}{4}{\Phi}(0) +O\left(\frac{\log\log Q}{\log Q}\right). 
}
Next, we consider the second sum in \eqref{eqn:afterapplyingexplicitformula}.  By GRH on $L(s, \sym^2(f))$ and the equation (4.23) in \cite{ILS}, we have

\es{\label{eqn:lambdaAtp^2}\frac{1}{N(Q)}\sum_q \Psi\bfrac{q}{Q}\frac{1}{\log q} \sumh_{f \in \mathcal H_{k}(q)}\sum_{p\nmid q} \frac{\lambda_f(p^2)\log p}{p} \widehat \Phi\left(\frac{\log p^2}{\log q} \right)  \ll \frac{\log \log Q}{\log Q}. }


Most of the work in this paper goes into estimating the first sum on the right-hand side of \eqref{eqn:afterapplyingexplicitformula}. We will prove the following.
\begin{prop} \label{prop:boundSigma_1} Assume GRH. Let $\Phi$ be an even Schwartz function with $\widehat \Phi$ compactly supported in $(-4, 4).$  Moreover, define 
$$ \Sigma_1 := \frac{1}{N(Q)}\sum_q \Psi\bfrac{q}{Q} \frac{1}{\log q}\sumh_{f \in \mathcal H_{k}(q)}\sum_{p\nmid q} \frac{\lambda_f(p) \log p}{\sqrt{p}} \widehat \Phi\left(\frac{\log p}{\log q} \right).$$
Then 
$$ \Sigma_1 \ll \frac{1}{\log Q}. $$
    
\end{prop}

\subsection{Proof of Proposition \ref{prop:main} and Theorem \ref{thm:main}} \label{sec:proofofMainTheorem}
From \eqref{eqn:afterapplyingexplicitformula}, \eqref{eqn:lambdaAtp^2} and Proposition~\ref{prop:boundSigma_1}, we have
\est{
\frac{1}{N(Q)}  \sum_q \Psi\bfrac{q}{Q} \frac{1}{\log q} \sumh_{f \in \mathcal H_{k}(q)} \sum_n \frac{\Lambda(n)c_f(n)}{\sqrt{n}} \widehat \Phi\left( \frac{\log n}{\log q}\right) =  -\frac{\Phi(0)}{4} + O\left( \frac{\log \log Q}{\log Q}\right).  
}
Taking $Q \rightarrow \infty$, we arrive at Proposition~\ref{prop:main}. 

To prove Theorem~\ref{thm:main}, we use the explicit formula from Lemma~\ref{lem:Explicit}, \eqref{eqn:asympN(q)}, and Proposition~\ref{prop:main} to deduce that
\begin{align*}
    \lim_{Q \rightarrow \infty} \mathscr {OL} (Q)   &= - \lim_{Q \rightarrow \infty}\frac{1}{N(Q)}\sum_{q} \Psi\bfrac{q}{Q}\frac{1}{\log q}  \sumh_{f \in \mathcal H_{k}(q)} \sum_{n=1}^{\infty}\frac{\Lambda(n)[c_f(n) + c_{\bar f}(n)]}{\sqrt{n}} \widehat\Phi\left(\frac{\log n}{\log q}\right) \\
    &+ \lim_{Q \rightarrow \infty} \frac{1}{N(Q)}\sum_{q} \Psi\bfrac{q}{Q}  \sumh_{f \in \mathcal H_{k}(q)} \left[\int_{-\infty}^{\infty} \Phi(x) \> dx +O_k\left(\frac{1}{\log q} \right) \right], \\
    &= 2\frac{ \Phi(0)}{4} + \int_{-\infty}^{\infty} \Phi(x) \> dx = \int_{-\infty}^{\infty} \Phi(x) \left( 1 + \frac{\delta_0}{2}\right) \> dx.
\end{align*}
This proves Theorem~\ref{thm:main}.

\section{Applying Petersson's formula}\label{sec:applypetersson}

The rest of the paper is devoted to proving Proposition~\ref{prop:boundSigma_1}. Applying Petersson's formula, Lemma~\ref{lem:PeterssonNg}, to $\Sigma_{1}$ defined in Proposition~\ref{prop:boundSigma_1}, we deduce that

\begin{align} 
\Sigma_1
& = \frac{1}{N(Q)}\sum_q \Psi\left(\frac{q}{Q}\right)   \frac{1}{\log q}\sum_{p\nmid q} \frac{ \log p}{\sqrt{p}} \widehat \Phi\left(\frac{\log p}{\log q} \right) \notag \\
& \ \ \ \ \ \times \sum_{\substack{q=L_1L_2d \\ L_1|q_1 \\ L_2|q_2}} \frac{\mu(L_1L_2)}{L_1L_2}  \prod_{\substack{p_1|L_1 \\ p_1^2 \nmid d}}  \left( 1-\frac{1}{p_1^2} \right)^{-1} \sum_{e|L_2^{\infty} }\frac{\Delta_d(e^2,p)}{e}, \label{eqn:Sigmaafterhecke}
\end{align}
where $p_1$ runs over the primes, and where $q=q_1q_2$, with $q_1$ the largest factor of $q$ satisfying $p_1|q_1 \Leftrightarrow p_1^2|q$ for all primes $p_1$.

\subsection{Truncation in $L_1L_2$}

We first bound the contribution of the range $L_1 L_2 \ge \mathcal{L}_0$ (equivalently $d \ll \frac{Q}{\mathcal L_0}$) for 
\begin{equation} \label{def:L_0}
\mathcal{L}_0  = (\log Q)^{6}.     \end{equation}

By \eqref{eqn:Sigmaafterhecke}, we need to bound
\begin{align*}
&\ll \frac{1}{Q \log Q} \sumthree_{\substack{L_1, L_2, d \\ L_1L_2 \ge \mathcal{L}_0}} \frac{1}{L_1L_2} \Psi\bfrac{L_1L_2d}{Q}  
\sum_{e|L_2^{\infty} } \frac{\tau(e^2)}{e}  \left| \sumh_{f\in B_k(d)} \sum_{p \nmid q} \frac{ \log p \lambda_f(p)} {\sqrt{p}} \widehat \Phi\left(\frac{\log p}{\log q} \right)\right| \notag \\
&\ll \frac{1}{Q \log Q} \sumthree_{\substack{L_1, L_2, d \\ L_1L_2 \ge \mathcal{L}_0}} \frac{1}{L_1L_2} \Psi\bfrac{L_1L_2d}{Q}  
\sum_{e|L_2^{\infty} } \frac{\tau(e^2)}{e} \log^{2 + \epsilon} q,
\end{align*}
upon applying Lemma~\ref{lem:CLee3.5} to the sum over $p$.  We remind the reader that we have used that $\lambda_f(p) = \lambda_g(p)$ for some Hecke newform $g$ of level dividing $d$.  The quantity above is bounded by
\begin{align}\label{eqn:truncatedbdd}
\frac{(\log Q)^{1 + \epsilon}}{Q} \sumthree_{\substack{L_1, L_2, d \\ d \ll Q/\mathcal{L}_0}} \frac{\tau(L_2)}{L_1L_2} \Psi\bfrac{L_1L_2d}{Q} \ll \frac{(\log Q)^{5}}{\mathcal{L}_0} \ll \frac{1}{\log Q}.
\end{align}

\subsection{Truncation in $e$}
We now bound the contribution of the range $e \ge E$ for 
\begin{equation} \label{def:E}
    E = (\log Q)^3.
\end{equation}

Again, we need to bound
\begin{align}\label{eqn:truncatee1}
&\frac{1}{Q \log Q}  \sumthree_{\substack{L_1, L_2, d \\ L_1L_2 \ll \mathcal{L}_0}} \frac{1}{L_1L_2} \Psi\bfrac{L_1L_2d}{Q}  
\sum_{\substack{e|L_2^{\infty} \\ e\ge E }}\frac{\tau(e)^2}{e}  \left| \sumh_{f\in B_k(d)} \sum_{p \nmid q} \frac{ \log p \lambda_f(p)} {\sqrt{p}} \widehat \Phi\left(\frac{\log p}{\log q} \right)\right| \notag \\
&\ll \frac{1}{Q \log Q}  \sumthree_{\substack{L_1, L_2, d \\ L_1L_2 \ll \mathcal{L}_0}} \frac{1}{L_1L_2} \Psi\bfrac{L_1L_2d}{Q}  
\sum_{\substack{e|L_2^{\infty} \\ e\ge E }} \frac{1}{e^{1-\epsilon}} (\log q)^{2 + \epsilon} ,
\end{align}
upon applying Lemma~\ref{lem:CLee3.5} to the sum over $p$.

We note that
\begin{align*}
\sum_{\substack{e|L_2^{\infty} \\ e\ge E }} \frac{1}{e^{1-\epsilon}}
\ll  \frac{1}{E^{1-2\epsilon}} \sum_{\substack{e|L_2^{\infty} }} \frac{1}{e^\epsilon}
\ll \frac{\tau(L_2)}{E^{1-2\epsilon}}, 
\end{align*}and so the quantity in \eqref{eqn:truncatee1} is bounded by

\begin{align}\label{eqn:truncateebdd}
\ll \frac{(\log Q)^{1 + \epsilon}}{Q E^{1-2\epsilon}}  \sumthree_{\substack{L_1, L_2, d \\ L_1L_2 \ll \mathcal{L}_0}} \frac{\tau(L_2)}{L_1L_2} \Psi\bfrac{L_1L_2d}{Q}  
\ll \frac{(\log Q)^{1 + \epsilon}}{E^{1-2\epsilon}} \ll \frac{1}{\log Q}.
\end{align}

\subsection{Filling in the sum over $p$}
In order to facilitate our application of Kuznetsov's formula later, we first add the terms corresponding to $p|q$. We show that this changes the value of $\Sigma_1$ by an amount that is not too large. To do this, we need to bound the sum
\begin{align*}
\mathcal{E} :=
& \frac{1}{N(Q)}\sum_q \Psi\bfrac{q}{Q}\frac{1}{\log q}\sum_{p|q} \frac{ \log p}{\sqrt{p}} \left|\widehat \Phi\left(\frac{\log p}{\log q} \right) \right|\\
& \times \sum_{\substack{q=L_1L_2d \\ L_1|q_1 \\ L_2|q_2}} \frac{1}{L_1L_2}  \prod_{\substack{p_1|L_1 \\ p_1^2 \nmid d}}  \left( 1-\frac{1}{p_1^2} \right)^{-1} \sum_{e|L_2^{\infty} }\frac{|\Delta_d(e^2,p)|}{e},
\end{align*}where we have ignored the conditions $e\le E$ and $L_1L_2 \le \mathcal{L}_0$ by positivity.  Since $(e, d)  =1$, Lemma~\ref{lem:petertruncate} implies that
\begin{align*}
\Delta_d(e^2,p)
& \ll_k  \frac{(ped)^{\epsilon}}{ d((p,d)+(e^2, d) )^{1/2}} \left( \frac{pe^2}{\sqrt{pe^2} +d} \right)^{1/2} \\
& \ll  \frac{(ped)^{\epsilon} }{ d}(pe^2)^{1/4} .
\end{align*}
Since $\Psi$ is non-negative,
\begin{align*}
\mathcal{E} \ll
& \frac{1}{N(Q)}\sum_q \Psi\bfrac{q}{Q} \frac{1}{\log q}\sum_{p|q} \frac{ \log p}{\sqrt{p}} \left| \widehat \Phi\left(\frac{\log p}{\log q} \right)\right| \\
& \times \sum_{\substack{q=L_1L_2d \\ L_1|q_1 \\ L_2|q_2}} \frac{1}{L_1L_2}  \prod_{\substack{p_1|L_1 \\ p_1^2 \nmid d}}  \left( 1-\frac{1}{p_1^2} \right)^{-1} \sum_{e|L_2^{\infty} }\frac{(ped)^{\epsilon} p^{1/4} e^{1/2} }{de}.
\end{align*}
Now
\begin{align*}
\sum_{e|L_2^{\infty} }e^{-\frac{1}{2}+\epsilon} = \prod_{p_1|L_2} \left(1 - p_1^{-\frac{1}{2}+\epsilon}\right)^{-1}
\ll L_2^{\epsilon}.
\end{align*}
Since $(pdL_2)^{\epsilon}\ll q^{\epsilon}$, we have
\begin{align*}
\mathcal{E} \ll
& \frac{1}{N(Q)}\sum_q \Psi\bfrac{q}{Q} \frac{1}{\log q}\sum_{p|q} \frac{ \log p}{\sqrt{p}} \left| \widehat \Phi\left(\frac{\log p}{\log q} \right)\right| \\
& \times \sum_{\substack{q=L_1L_2d \\ L_1|q_1 \\ L_2|q_2}} \frac{1}{L_1L_2}  \prod_{\substack{p_1|L_1 \\ p_1^2 \nmid d}}  \left( 1-\frac{1}{p_1^2} \right)^{-1}  p^{\frac{1}{4}+\epsilon} d^{-1+\epsilon} L_2^{\epsilon} \\
&\ll   \frac{Q^\epsilon}{N(Q)}\sum_q \Psi\bfrac{q}{Q} \frac{1}{\log q}\sum_{p|q} \frac{ \log p}{p^{1/4}}  \sum_{\substack{q=L_1L_2d \\ L_1|q_1 \\ L_2|q_2}} \frac{1}{L_1L_2 d} \\
&\ll  \frac{Q^\epsilon}{N(Q)}\sum_q \Psi\bfrac{q}{Q} \frac{1}{\log q} (\log q) \frac{q^{\epsilon}}{q} \ll \frac{1}{Q^{1 - \epsilon}}
\end{align*}
where we used $N(Q) \asymp Q$ (see \eqref{eqn:asympN(q)}). It follows from this, \eqref{eqn:afterapplyingng}, \eqref{eqn:truncatedbdd} and \eqref{eqn:truncateebdd} that
\begin{align} \label{eqn:Sigma_1toSigma_2}
\Sigma_1 = \Sigma_{2} + O\left( \frac{1}{\log Q} \right),
\end{align} where
\es{ \label{eqn:afterapplyingpetersson}
\Sigma_{2}
& = \frac{1}{N(Q)}\sum_q \Psi\left(\frac{q}{Q}\right)   \frac{1}{\log q}\sum_{p} \frac{ \log p}{\sqrt{p}} \widehat \Phi\left(\frac{\log p}{\log q} \right) \\
& \ \ \ \ \ \times \sum_{\substack{q=L_1L_2d \\ L_1|q_1 \\ L_2|q_2 \\ L_1L_2 < \mathcal{L}_0}} \frac{\mu(L_1L_2)}{L_1L_2}  \prod_{\substack{p_1|L_1 \\ p_1^2 \nmid d}}  \left( 1-\frac{1}{p_1^2} \right)^{-1} \sum_{\substack{e|L_2^{\infty} \\ e < E}}\frac{1}{e} \\
& \ \ \ \ \ \times  (2\pi i)^{-k} \sum_{c\geq 1} \frac{S(e^2,p;cd)}{cd}J_{k-1} \left( \frac{4\pi \sqrt{pe^2}}{cd}\right).  
}

\section{Setup to bound $\Sigma_2$ and the proof of Proposition~\ref{prop:boundSigma_1}} \label{sec:applykutz}

By the Remark below Lemma~\ref{lem:PeterssonNg}, we may replace the conditions $L_1|q_1$ and $L_2|q_2$ in \eqref{eqn:afterapplyingpetersson} by $L_1|d$ and $(L_2,d)=1$. Therefore 
\begin{align*} 
\Sigma_{2}
& = \frac{1}{N(Q)}\sum_q \Psi\left(\frac{q}{Q}\right)   \frac{1}{\log q}\sum_{p} \frac{ \log p}{\sqrt{p}} \widehat \Phi\left(\frac{\log p}{\log q} \right)\\
& \ \ \ \ \ \times \sum_{\substack{q=L_1L_2d \\ L_1|d \\ (L_2,d)=1 \\ L_1L_2 < \mathcal L_0}} \frac{\mu(L_1L_2)}{L_1L_2}  \prod_{\substack{p_1|L_1 \\ p_1^2 \nmid d}}  \left( 1-\frac{1}{p_1^2} \right)^{-1} \sum_{\substack{e|L_2^{\infty} \\ e < E}}\frac{1}{e} \\
& \ \ \ \ \ \times  (2\pi i)^{-k} \sum_{c\geq 1} \frac{S(e^2,p;cd)}{cd}J_{k-1} \left( \frac{4\pi \sqrt{pe^2}}{cd}\right).
\end{align*}
Since $L_1|d$, we write $d = L_1 m$ and have
\begin{align*}
\prod_{\substack{p_1|L_1 \\ p_1^2 \nmid d}}  \left( 1-\frac{1}{p_1^2} \right)^{-1}
& = \prod_{ p_1|L_1  }  \left( 1-\frac{1}{p_1^2} \right)^{-1}  \prod_{\substack{p_1|L_1 \\ p_1  | m}}  \left( 1-\frac{1}{p_1^2} \right) \\
& = \Bigg\{\prod_{ p_1|L_1  }  \left( 1-\frac{1}{p_1^2} \right)^{-1}\Bigg\} \Bigg\{ \sum_{r|(L_1,m)}\frac{\mu(r)}{r^2}\Bigg\}.
\end{align*}
Using this and substituting  $q=L_1L_2d = L_1^2 L_2m$, where $(m, L_2) = 1$, in the above expression for $\Sigma_2$ gives 
\begin{align*}
\Sigma_{2}
= & \frac{(2\pi i)^{-k}}{N(Q)}\sum_{\substack{L_1,L_2  \\ L_1L_2 < \mathcal L_0}} \frac{\mu(L_1L_2)}{L_1L_2}  \prod_{ p_1|L_1  }  \left( 1-\frac{1}{p^2} \right)^{-1} \sum_{p } \frac{ \log p}{\sqrt{p}}   \sum_{\substack{e|L_2^{\infty} \\ e < E} }\frac{1}{e} \sum_{c\geq 1}  \\
& \times \sum_{\substack{m  \\ (m,L_2)=1}} \sum_{r|(L_1,m)}\frac{\mu(r)}{r^2}\frac{S(e^2,p;cL_1m)}{cL_1m} \Psi\left(\frac{L_1^2L_2m}{Q}\right) \\
& \times \frac{1}{\log (L_1^2L_2m)}\widehat \Phi\left(\frac{\log p}{\log (L_1^2L_2m)} \right)J_{k-1} \left( \frac{4\pi \sqrt{pe^2}}{cL_1m}\right).
\end{align*}
Since $r|m$, we may substitute $m=rn$ to write
\begin{align*}
\Sigma_{2}
= & \frac{(2\pi i)^{-k}}{N(Q)}\sum_{\substack{L_1,L_2  \\ L_1L_2 < \mathcal L_0}} \frac{\mu(L_1L_2)}{L_1L_2}  \prod_{ p_1|L_1  }  \left( 1-\frac{1}{p^2} \right)^{-1} \sum_{r|L_1}\frac{\mu(r)}{r^2} \sum_{p } \frac{ \log p}{\sqrt{p}}    \\
& \times \sum_{\substack{e|L_2^{\infty} \\ e < E} }\frac{1}{e} \sum_{c\geq 1}  \sum_{\substack{n  \\ (n,L_2)=1}} \frac{S(e^2,p;cL_1rn)}{cL_1rn} \Psi\left(\frac{L_1^2L_2rn}{Q}\right) \\
& \times \frac{1}{\log (L_1^2L_2rn)}\widehat \Phi\left(\frac{\log p}{\log (L_1^2L_2rn)} \right)J_{k-1} \left( \frac{4\pi \sqrt{pe^2}}{cL_1rn}\right).
\end{align*}
Next, we use M\"{o}bius inversion to detect the condition $(n,L_2)=1$ and deduce that
\es{ \label{eqn:afterapplyingpetersson2}
\Sigma_{2}
= & \frac{(2\pi i)^{-k}}{N(Q)}\sum_{\substack{L_1,L_2 \\ L_1L_2 < \mathcal L_0}} \frac{\mu(L_1L_2)}{L_1L_2}  \prod_{ p_1|L_1  }  \left( 1-\frac{1}{p^2} \right)^{-1} \sum_{r|L_1}\frac{\mu(r)}{r^2} \sum_{d|L_2}\mu(d) \sum_{p } \frac{ \log p}{\sqrt{p}}     \\
& \times \sum_{\substack{e|L_2^{\infty} \\ e < E} }\frac{1}{e} \sum_{c\geq 1}  \sum_{n}  \frac{S(e^2,p;cL_1rdn)}{cL_1rdn} \Psi\left(\frac{L_1^2L_2rdn}{Q}\right)  \\
& \times \frac{1}{\log (L_1^2L_2rdn)}\widehat \Phi\left(\frac{\log p}{\log (L_1^2L_2rdn)} \right)J_{k-1} \left( \frac{4\pi \sqrt{pe^2}}{cL_1rdn}\right). 
}

\subsection{Applying Kutznetsov's formula}
The $n$-sum in \eqref{eqn:afterapplyingpetersson2} is
\begin{align*}
\sum_{n}  \frac{S(e^2,p;cL_1rdn)}{cL_1rdn} \Psi\left(\frac{L_1^2L_2rdn}{Q}\right) \frac{1}{\log (L_1^2L_2rdn)}\\
\widehat \Phi\left(\frac{\log p}{\log (L_1^2L_2rdn)} \right)J_{k-1} \left( \frac{4\pi \sqrt{pe^2}}{cL_1rdn}\right).
\end{align*}
We let $\fm= cL_1rdn$ and write this sum as
\begin{align*}
& \sum_{\substack{\fm \geq 1 \\ \fm\equiv 0 \bmod{cL_1rd}}} \frac{S(e^2,p;\fm)}{\fm} \Psi\left(\frac{L_1L_2\fm}{cQ}\right) \frac{1}{\log (L_1L_2\fm/c)}\\
& \hspace{1in} \times \widehat \Phi\left(\frac{\log p}{\log (L_1L_2\fm/c)} \right)J_{k-1} \left( \frac{4\pi \sqrt{pe^2}}{\fm}\right) \\
& = \sum_{\substack{\fm \geq 1 \\ \fm\equiv 0 \bmod{cL_1rd}}} \frac{S(e^2,p;\fm)}{\fm} f\left( 4\pi \frac{\sqrt{pe^2}}{\fm} \right),
\end{align*}
where $f(\cdot) = f_{c,L_1,L_2,Q,p,e,k}(\cdot)$ is defined by
\begin{align*}
f\left( 4\pi \frac{\sqrt{pe^2}}{\fm} \right) : = \Psi\left(\frac{L_1L_2\fm}{cQ}\right) \frac{1}{\log (L_1L_2\fm/c)} \widehat \Phi\left(\frac{\log p}{\log (L_1L_2\fm/c)} \right)J_{k-1} \left( \frac{4\pi \sqrt{pe^2}}{\fm}\right), 
\end{align*}
i.e., 
\begin{align}\label{eqn:fxi}
f(\xi) : = \Psi\left(\frac{4\pi L_1L_2\sqrt{pe^2}}{cQ\xi}\right) \frac{1}{\log \left(\frac{4\pi L_1L_2\sqrt{pe^2}}{c\xi}\right)} \widehat \Phi\left(\frac{\log p}{\log \left(\frac{4\pi L_1L_2\sqrt{pe^2}}{c\xi}\right)} \right)J_{k-1}(\xi). 
\end{align}

We introduce a smooth partition of unity for the sum over $p$, and write
$$\sum_p \textup{...} = \sumd_P \sum_p V\bfrac{p}{P}\textup{...},
$$where $\sumd_P$ denotes a sum over $P = 2^j$ for $j\ge 0$, and $V$ is a smooth function compactly supported on $[1/2, 3]$ satisfying $\sumd_P V\bfrac{p}{P} = 1$ for all $p \geq 1$. 
When studying the contribution of $\sum_p V\bfrac{p}{P}...$, we need to separate the dependence of $f(\xi)$ on $p$.  To this end, we let

\begin{equation} \label{eqn: Xdefinition}
X := \frac{4\pi L_1 L_2 \sqrt{Pe^2}}{cQ},
\end{equation}and

\begin{align}\label{eqn:Hdef}
H(\xi, \lambda) := \Psi\left(\frac{X}{\xi} \sqrt{\lambda}\right) \frac{\log Q}{\log \left(\frac{X}{\xi} \sqrt{\lambda} Q\right)} \widehat \Phi\left(\frac{\log \lambda + \log P}{\log \left( \frac{X}{\xi} \sqrt{\lambda} Q \right)} \right).
\end{align}

We let 
$$\widehat{H}(u, v) = \int_{-\infty}^\infty \int_{-\infty}^\infty H(\xi, \lambda) \e{-\xi u - \lambda v} d\xi d\lambda$$ be the usual Fourier transform of $H$.  For reference later, we record the following bounds on $\hat H$.

\begin{lem}\label{lem:hatHbdd}
With notation as above, we have that for any $A>0$, 
\begin{equation*}
\widehat{H}(u, v) \ll_A \bfrac{1}{(1+|u|)(1+|v|)}^A.
\end{equation*}
\end{lem}

\begin{proof}
By the support of $\Psi$, we have that $\frac{X}{\xi} \sqrt{\lambda} \asymp 1$, and a little calculation gives that
\begin{align*}
\frac{\partial^{l}}{\partial^l \xi} \frac{\partial^{r}}{\partial^r \lambda} H(\xi, \lambda) \ll \frac{1}{\xi^l \lambda^r}.
\end{align*}
The lemma then follows from repeated integration by parts.
\end{proof}

The function $\Psi$ in \eqref{eqn:Hdef} is compactly supported. Thus $H(\xi, \frac pP)$ is compactly supported in $\xi$ since $p \asymp P$. Let us say it is supported in $(a_1, b_1)$.  By \eqref{eqn:fxi} and \eqref{eqn:Hdef}, we have that
\begin{equation*}
f(\xi) = \frac{1}{\log Q} H\left(\xi, \frac pP\right) J_{k-1}(\xi) = \frac{1}{\log Q} H\left(\xi, \frac pP\right) J_{k-1}(\xi) W\left( \frac{\xi}{X} \right),
\end{equation*}
where $W$ is a smooth function that is compactly supported in $(a, b), $ $0<a<a_1$, $b_1 < b$, such that $W(\xi) = 1$ when $\xi \in [a_1, b_1]$. 

By Fourier inversion, we have
\begin{align*}
f(\xi) = \frac{1}{\log Q}  J_{k-1}(\xi)  W\left( \frac{\xi}{X} \right) \int_{-\infty}^\infty \int_{-\infty}^\infty  \widehat{H}(u, v) \e{u\xi+ v\frac{p}{P}} du dv.
\end{align*}
For notational convenence, let
\begin{equation*}
h_u(\xi) = J_{k-1}(\xi)W\bfrac{\xi}{X} \e{u\xi}.
\end{equation*}
We have that
\es{ \label{eqn:Sigma2afterFouriertransform}
\Sigma_{2} 
&=  \frac{(2\pi i)^{-k}}{N(Q)}\sum_{\substack{L_1,L_2 \\ (L_1,L_2)=1 \\ L_1L_2 < \mathcal L_0}} \frac{\mu(L_1L_2)}{L_1L_2}  \prod_{ p_1|L_1  }  \left( 1-\frac{1}{p^2} \right)^{-1} \sum_{r|L_1}\frac{\mu(r)}{r^2} \sum_{d|L_2}\mu(d) \sumd_P   \\
&\times  \frac{1}{\log Q}  \sum_{\substack{e|L_2^{\infty} \\ e < E} }\frac{1}{e}   \int_{-\infty}^\infty \int_{-\infty}^\infty  \widehat{H}(u, v) \sum_{c\geq 1} \sum_{p }  \frac{ \log p}{\sqrt{p}} V\bfrac{p}{P}  \e{v\frac{p}{P}} \mathcal{S}(u, p) du dv,
}
where
\begin{align*} 
\mathcal S(u, p) := 
\sum_{n}  \frac{S(e^2,p;cL_1rdn)}{cL_1rdn} h_u\bfrac{4\pi \sqrt{pe^2}}{cL_1rdn}.
\end{align*}
We apply Kutznetsov's formula from Lemma~\ref{lem:kuznetsov} to  $\mathcal S(u, p)$ and obtain that
\es{ \label{eqn:sumoverpseparateintoDisCtnHol}
\sum_{c\geq 1} \sum_{p } &\frac{ \log p}{\sqrt{p}}  \e{v\frac{p}{P}} V\bfrac{p}{P}  \mathcal{S}(u, p)  
\\
&=  \sum_{c\geq 1}\sum_{p } \frac{ \log p}{\sqrt{p}} \e{v\frac{p}{P}} V\bfrac{p}{P} \left[\Dis(c,p; u) + \Ctn(c,p; u) + \Hol(c,p; u)\right]
}
where
\begin{align*}
\textup{Dis}(c,p; u) &:=  \sum_{j=1}^{\infty} \frac{\overline{\rho_j}(e^2)\rho_j(p) \sqrt{pe^2} }{\cosh(\pi \kappa_j)} h_+(\kappa_j) \\
\Ctn(c,p; u) &:= \frac{1}{\pi} \sum_{\mathfrak{c}} \int_{-\infty}^{\infty}  \frac{\sqrt{pe^2}}{\cosh(\pi t)} \overline{\varphi_{\mathfrak{c}}} (e^2, t)  \varphi_{\mathfrak{c} } (p, t) h_+ (t) \,dt \textup{, \;\;\;and}\\
\Hol(c,p; u) &:= \frac{1}{2\pi} \sum_{ \substack{\ell\geq 2 \mbox{\scriptsize{ even}} \\ 1 \leq j \leq \theta_{\ell}(cL_1rd)} } (\ell - 1)! \sqrt{pe^2} \, \overline{\psi_{j,\ell}}(e^2) \psi_{j,\ell} (p) h_h(\ell),
\end{align*}where it is implied that $e, r, d, L_1, L_2$ are fixed. Note that these forms are of level $cL_1rd.$ The bound for $\Sigma_2$ will follow from the following Propositions.

\begin{prop}\label{prop:DisHol}
With notation as above, we have
\begin{align*}
\sum_{c\geq 1} \sum_{p  } \frac{ \log p}{\sqrt{p}} \e{v\frac{p}{P}}\left(\Dis(c,p; u) + \Hol(c,p; u)\right)  V\left( \frac pP\right)\ll Q^{\epsilon}(1 + |u|)^{2}  (1 + |v|)^2 \frac{\sqrt P}{Q}.
\end{align*}
\end{prop}

\begin{prop}\label{prop:Ctn}
With notation as above, we have
\begin{align*}
\sum_{c \geq 1} \sum_{p  } \frac{ \log p}{\sqrt{p}}\e{v\frac{p}{P}} \Ctn(c,p; u) V\left( \frac pP\right) \ll Q^{\epsilon}  \ll Q^{\epsilon} (1 + |u|)^2  (1 + |v|)^2\left(P^{\frac 14 + \epsilon} + \frac{\sqrt P}{Q}\right).
\end{align*}
\end{prop}

\subsection{Proof of Proposition \ref{prop:boundSigma_1}}
From \eqref{eqn:Sigma2afterFouriertransform}, \eqref{eqn:sumoverpseparateintoDisCtnHol}, Lemma~\ref{lem:hatHbdd}, Proposition~\ref{prop:DisHol} and Proposition~\ref{prop:Ctn}, we deduce that 
\est{\Sigma_2 &\ll \frac{Q^{\epsilon}}{N(Q) \log Q} \sum_{\substack{L_1,L_2 \\ L_1L_2 < \mathcal L_0}} \frac{1}{L_1L_2}  \prod_{ p_1|L_1  }  \left( 1-\frac{1}{p^2} \right)^{-1} \sum_{r|L_1}\frac{1}{r^2} \sum_{d|L_2}1 \sum_{\substack{e|L_2^{\infty} \\ e < E} }\frac{1}{e}   \\
&\times  \sumd_P   \int_{-\infty}^\infty \int_{-\infty}^\infty \frac{1}{(1 + |u|)^{10} (1 + |v|)^{10}}  (1 + |v|)^2(1 + |u|)^2  \left(P^{\frac 14 + \epsilon} + \frac{\sqrt P}{Q}\right) du dv. }
Recall that $\mathcal L_0 = (\log Q)^6$ from \eqref{def:L_0}, $E = (\log Q)^3$ from \eqref{def:E}, $N(Q) \asymp Q$ from \eqref{eqn:asympN(q)}, and $P \ll Q^{4 - \delta}$ for some fixed $\delta>0$ due to the compact support of $\widehat \Phi$ in $(-4, 4).$ Thus 
\est{\Sigma_2 &\ll \frac{Q^{\epsilon}}{Q(\log Q)} \left( Q^{(4 - \delta)(1/4 + \epsilon)} + \frac{Q^{2 - \delta/2}}{Q} \right) \sum_{\substack{L_1, L_2 \\ L_1L_2 < \mathcal L_0}  }\frac{\tau(L_1)\tau^2(L_2)}{L_1L_2} \\
&\ll Q^{-{\delta}/8}}
upon choosing sufficiently small $\epsilon$. From this and \eqref{eqn:Sigma_1toSigma_2}, we arrive at
$$
\Sigma_1 \ll \frac{1}{\log Q},
$$
and this completes the proof of Proposition~\ref{prop:boundSigma_1}.

\section{Proof of Proposition \ref{prop:DisHol}} \label{sec:dispart}
We will bound 
\begin{align*}
\mathcal {DISC} = \sum_{c \geq 1} \sum_{p  } \frac{ \log p}{\sqrt{p}}  \e{v\frac{p}{P}}\Dis(c,p; u)  V\left( \frac pP\right),
\end{align*}
the bound for the contribution of $\Hol(c, p)$ being similar but easier (since Ramanujan-Petersson is known). From \eqref{eqn: Xdefinition}, we remind the reader that
\begin{equation*}
X = \frac{4\pi L_1L_2 \sqrt{Pe^2}}{cQ},
\end{equation*}and then $h_u(\xi)$ is supported on an interval of the form $(aX, bX)$ for some $b>a>0$. From \eqref{eqn:sumoverpseparateintoDisCtnHol}, we now write
\es{ \label{eqn:Discsum}
\mathcal {DISC}
&= \sum_{c \geq 1} \sum_{j=1}^{\infty} \frac{e \overline{\rho_j}(e^2)}{\cosh(\pi \kappa_j)} h_+(\kappa_j) \sum_{p }\frac{ {
\sqrt{p}\rho_j}(p) \log p}{\sqrt{p}} \e{\frac{vp}{P}} V\left( \frac pP\right)\\
&= S_1 + S_2 
} where $S_1$ is the contribution of real $\kappa_j ,$  and $S_2$ is the contribution of imaginary $\kappa_j$ corresponding to exceptional eigenvalues.

We recall that in the sum above, $\sum_{j=1}^{\infty} $ denotes a sum over the spectrum of level $cL_1rd$, where $\{u_j\}_{j=1}^\infty$ is the orthonormal basis for the Maass forms of level $cL_1rd$ described in \S\ref{subsec:newformoldform}, and $\rho_j(n)$ denotes the Fourier coefficients of $u_j$.  We recall that $u_j$ is of the form $f^{(g)}$ where $f$ is a Hecke newform of level $M$ with $M |cL_1rd$, and $g|\frac {cL_1rd}{M}$.  

\begin{lem}\label{lem:primesumMaassbdd}
With notation as above, we have
\begin{align*}
\sum_{p }\frac{ \sqrt{p} {\rho_j}(p) \log p}{\sqrt{p}} \e{\frac{vp}{P}} V\left( \frac pP\right) \ll  (|\rho_j(1)|+|\rho_f(1)|) (cL_1rd)^\epsilon (\log (P+2))^{1 + \epsilon} (1+|v|)^{2},
\end{align*}
where $f$ is the Hecke newform of some level $M|cL_1rd$ such that $u=f^{(g)}$ for some $g|\frac{cL_1rd}{M}$.
\end{lem}

\begin{proof}
By Mellin inversion, we have that
\begin{align}\label{eqn:primsumMaass}
&\sum_{p }\frac{ \sqrt{p} {\rho_j}(p) \log p}{\sqrt{p}} \e{\frac{vp}{P}} V\left( \frac pP\right) \notag \\
=&\frac{1}{2\pi i} \sum_{p}\frac{ \sqrt{p}{\rho_j}(p) \log p}{\sqrt{p}} V_0\bfrac{p}{P} \int_{(0)} p^{-s} \widetilde{\mathcal W}(s) P^s ds  
\end{align}
where $\widetilde{\mathcal W}(s)$ is the usual Mellin transform of $\mathcal{W}_v(x) := \mathcal W(x) := \e{vx} V(x)$ and $V_0(x)$ is chosen to be a smooth function that is compactly supported on $(0, \infty)$ such that $V_0(x) = 1$ whenever $V(x) \neq 0$.  We have that 
$$\W^{(j)}(x) \ll 1+|v|^j,
$$for all $j$, whence by integration by parts
$$
\widetilde{\W}(it) \ll_A \bfrac{1+|v|}{1+|t|}^A
$$
for any $A>0$. Writing $s = it$, we have by Lemma~\ref{lem:heckebasisfouriersumbdd} that
\begin{align*}
&\sum_{p}\frac{\sqrt{p} \rho_j(p)\log(p)}{p^{1/2+it}}  V_0\left(\frac pP\right) \\
&\ll |\rho_j(1)|  \log(cL_1rd+|t| )(\log (P+2))^{1 + \epsilon}  + |\rho_f(1)| (cL_1rd)^\epsilon\\
&\ll (|\rho_j(1)|+|\rho_f(1)|)  (cL_1rd)^\epsilon (1+|t|)^\epsilon (\log (P+2))^{1 + \epsilon},
\end{align*}
where $f$ is the Hecke newform of some level $M$ such that $u_j=f^{(g)}$ with $M|cL_1rd$ and $g|\frac{cL_1rd}{M}$. Hence the quantity in \eqref{eqn:primsumMaass} is 
\begin{align*}
&\ll (|\rho_j(1)|+|\rho_f(1)|) (cL_1rd)^\epsilon (\log (P+2))^{1 + \epsilon} (1+|v|)^{2} \int_{-\infty}^\infty (1+|t|)^\epsilon  \frac{dt}{(1+|t|)^{2}}  \\
&\ll (|\rho_j(1)|+|\rho_f(1)|) (cL_1rd)^\epsilon (\log (P+2))^{1 + \epsilon} (1+|v|)^{2}.
\end{align*}
\end{proof}

By Equation \eqref{eqn:fouriercoeffbdd}, we have
\begin{equation}\label{eqn:bddforerhoe2}
 e \overline{\rho_j}(e^2) \ll (cL_1rd)^{\epsilon} e^{1 + \epsilon} |\rho_f(1)|. 
\end{equation}
From Lemma~\ref{lem:boundforhu}~(1), \eqref{eqn:Discsum}, Lemma~\ref{lem:primesumMaassbdd}, and \eqref{eqn:bddforerhoe2}, we deduce that
\est{S_1 \ll &\sum_{c } \min\left\{X^{k-1}, \frac 1{\sqrt X} \right\} (cL_1rd)^{\epsilon} (\log (P+2))^{1+\epsilon} (1 + |v|)^2 e^{1 + \epsilon} \\
&\times \sum_{j = 1}^{\infty}  \frac{(|\rho_j(1)|+|\rho_f(1)|) |\rho_f(1)|}{\cosh(\pi \kappa_j)} \frac{(1 + |\log X|) }{F^{1- \epsilon}} \left( \frac{F}{1 + |\kappa_j|}\right)^C ,    }
where $f$ is the Hecke newform of level $M$ such that $u_j=f^{(g)}$ with $M|cL_1rd$ and $g|\frac{cL_1rd}{M}$. We may write
$$
|\rho_j(1)| |\rho_f(1)| \leq \frac{1}{2} |\rho_j(1)|^2 + \frac{1}{2} |\rho_f(1)|^2.
$$
Each newform $f$ appears at most $\ll (cL_1rd)^{\epsilon}$ times in the above $j$-sum, since the number of $u_j$ with $u_j=f^{(g)}$ for some $g$ is equal to the number of divisors $g$ of $\frac{cL_1rd}{M}$. Moreover, $f$ has the same eigenvalues as $f^{(g)}$ for each $g$. Thus, since $\rho_f(1) = \rho_{f^{(1)}}(1)$ and by positivity, we see that
\begin{align*}
\sum_{j = 1}^{\infty}\frac{|\rho_f(1)|^2}{\cosh(\pi \kappa_j)} \frac{(1 + |\log X|)}{F^{1- \epsilon}} \left( \frac{F}{1 + |\kappa_j|}\right)^C \ll (cL_1rd)^{\epsilon} \sum_{\ell = 1}^{\infty}\frac{|\rho_{\ell}(1)|^2}{\cosh(\pi \kappa_{\ell})} \frac{(1 + |\log X|)}{F^{1- \epsilon}} \left( \frac{F}{1 + |\kappa_{\ell}|}\right)^C,
\end{align*}
and hence we have
\est{S_1 \ll & \sum_{c } \min\left\{X^{k-1}, \frac 1{\sqrt X} \right\}  (cL_1rd)^{\epsilon} (\log (P+2))^{1 + \epsilon} (1 + |v|)^2 e^{1 + \epsilon} \\
&\times \sum_{j = 1}^{\infty}  \frac{|\rho_j(1)|^2}{\cosh(\pi \kappa_j)} \frac{(1 + |\log X|)}{F^{1- \epsilon}} \left( \frac{F}{1 + |\kappa_j|}\right)^C.  }
The sum over the spectrum is bounded by 
\begin{equation*}
\sum_{|k_j| < F}  \frac{|\rho_j(1)|^2}{\cosh(\pi \kappa_j)} \frac{1}{F^{1- \epsilon}}  +  \sum_{|k_j| \geq F}  \frac{|\rho_j(1)|^2}{\cosh(\pi \kappa_j)} \frac{1}{F^{1- \epsilon}}  \left( \frac{F}{1 + |\kappa_j|}\right)^{2 + \epsilon}  \ll F^{1 + \epsilon},
\end{equation*}
where we have used the spectral large sieve bound
\begin{equation*}
\sum_{|\kappa_j| \le K} \frac{1}{\cosh(\pi \kappa_j)} |\rho_j(1)|^2 \ll K^2.
\end{equation*}
from Deshoulliers and Iwaniec~\cite{DI}. Recalling that $F < (1 + |u|)(1 + X)$ and $X = \frac{L_1L_2\sqrt{Pe^2}}{cQ}$, we arrive at 
\es{ \label{eqn:S1boundXsmall} S_1 
&\ll \sum_{c}  \min \left\{ X^{k - 1}, \frac 1{\sqrt X}\right\}(1 +|\log X|)(1 + X)^{1 + \epsilon} (1 + |u|)^{1 + \epsilon} (cL_1rd)^{\epsilon} (\log (P+2))^{1 + \epsilon} (1 + |v|)^2 e^{1 + \epsilon} \\
&\ll \left(\sum_{c \gg \frac{L_1L_2\sqrt{Pe^2}}{Q}} X^{k-1}(1 + |\log X|) + \sum_{c \ll \frac{L_1L_2\sqrt{Pe^2}}{Q}} \frac {1}{\sqrt X} X^{1 + \epsilon} (1 + \log X)\right) \\
& \ \ \ \hskip 1in \times (1 + |u|)^{1 + \epsilon} (cL_1rd)^{\epsilon} (\log P)^{1 + \epsilon} (1 + |v|)^2 e^{1 + \epsilon} \\
&\ll Q^{\epsilon}(1 + |u|)^{2}  (1 + |v|)^2 \frac{\sqrt P}{Q}, }
where we have used that $k \geq 3$, $L_1L_2, e \ll Q^{\epsilon}$ (see \eqref{def:L_0} and \eqref{def:E}),  $ r | L_1, d| L_2, $ and $\log (P+2) \ll \log Q.$

Similarly, \eqref{lem:boundforhu} (2), \eqref{eqn:Discsum}, Lemma~\ref{lem:primesumMaassbdd} and \eqref{eqn:bddforerhoe2} give
\es{\label{eqn:S2boundXsmall} S_2 &\ll \sum_{c}  \min \left\{ X^{k - 1}, \frac 1{\sqrt X}\right\}  (cL_1rd)^{\epsilon} (\log (P+2))^{1 + \epsilon}(1 + |v|)^2 e^{1 + \epsilon} \\
& \hskip 2in \times \sum_{\substack{\kappa_j = ir \\ |r| \leq 1/4}}  \frac{|\rho_j(1)|^2}{\cosh(\pi \kappa_j)} ((1+ |u|)^{1/2} + X^{-1/2}) \\
&\ll \sum_{c } \min \left\{ X^{k - 1}, \frac 1{\sqrt X}\right\} (cL_1rd)^{\epsilon} (\log (P+2))^{1 + \epsilon} (1 + |v|)^2 e^{1 + \epsilon}((1+ |u|)^{1/2} + X^{-1/2})  \\
&\ll  Q^{\epsilon}(1 + |u|)^{2}  (1 + |v|)^2 \frac{\sqrt P}{Q}.}

From \eqref{eqn:Discsum} and the bounds of $S_1$ in \eqref{eqn:S1boundXsmall} and $S_2$ in \eqref{eqn:S2boundXsmall}, we arrive at Proposition~\ref{prop:DisHol}.

\section{Proof of Proposition \ref{prop:Ctn}}\label{sec:ctn}

\subsection{Eisenstein series}

To prove Proposition~\ref{prop:Ctn}, we use the set of representatives for the cusps of $\Gamma_0(q)$ given by the following result due to Kiral and Young~\cite{KY}.

\begin{lem}\cite[Corollary~3.2]{KY}\label{lem:KiralYoungcor}
Let $N$ be a positive integer. A complete set of representatives for the set of inequivalent cusps of $\Gamma_0(N)$ is given by $(ab)^{-1}$ where $b$ runs over divisors of $N$, and $a$ runs modulo $(b,N/b)$, coprime to $(b,N/b)$. We may choose the $a$ so that $(a,N)=1$ by adding a suitable multiple of $(b,N/b)$.
\end{lem}

We need an explicit expression for the Fourier coefficients of Eisenstein series. Theorem 3.4 of Kiral and Young's work \cite{KY} leads to the following Lemma.

\begin{lem}\label{lem:KiralYounglem}
Let $\mathfrak{c}$ be a cusp of $\Gamma_0(N)$, and let $\mathfrak{c}=(ab)^{-1}$ be its representative given by Lemma~\ref{lem:KiralYoungcor}. Suppose that
\begin{equation*}
b=b_0b',
\end{equation*}
where $b_0$ is the largest factor of $b$ that is relatively prime to $N/b$.
\begin{enumerate}
\item[\upshape{(a)}] If $p\nmid N$, then $\varphi_{\mathfrak{c}} (p,t)=0$ unless
\begin{equation}\label{eqn: factorb}
b' = (b,N/b),
\end{equation}
in which case
\begin{align*}
\varphi_{\mathfrak{c}} (p, t)
& = \frac{\pi^{\frac{1}{2}+it}}{\Gamma(\frac{1}{2}+it)} \frac{\mu(b_0)}{(Nb_0)^{\frac{1}{2}+it}} p^{-\frac{1}{2}-it} \frac{1}{\varphi( b' )}  \sum_{\chi \bmod b' }  \frac{\chi(-\overline{p}\overline{b_0}a) \tau(\overline{\chi}) }{L(1+2it,\overline{\chi^2}\chi_0)} \\
& \ \ \ \ + \frac{\pi^{\frac{1}{2}+it}}{\Gamma(\frac{1}{2}+it)} \frac{\mu(b_0)}{(Nb_0)^{\frac{1}{2}+it}}p^{-\frac{1}{2}+it}\frac{1}{\varphi( b' )}  \sum_{\chi \bmod b' }  \frac{\chi(-p\overline{b_0}a) \tau(\overline{\chi}) }{L(1+2it,\overline{\chi^2}\chi_0)},
\end{align*}
where $\varphi(b')$ is the Euler totient function, $\tau(\overline{\chi})$ is the Gauss sum, and $\chi_0$ is the principal character modulo $N/b$. Also, if $e$ is a positive integer and \eqref{eqn: factorb} holds, then
\begin{align*}
\varphi_{\mathfrak{c}} (e^2, t) =
& \frac{\pi^{\frac{1}{2}+it}}{\Gamma(\frac{1}{2}+it)} e^{-1+2it} \frac{1}{(Nb_0)^{\frac{1}{2}+it}}  S(\tfrac{e^2}{(e^2,b')},0;b_0) \\
& \times \sum_{\substack{d|e^2 \\ (d,N/b)=1}} d^{-2it} \frac{1}{\varphi(\frac{b'}{(e^2,b')})} \sum_{\chi \bmod \frac{b'}{(e^2,b')} }  \frac{\chi(-\tfrac{e^2}{(e^2,b')}\overline{b_0d^2}a) \tau(\overline{\chi}) }{L(1+2it,\overline{\chi^2}\chi_0)}. 
\end{align*}
\item[\upshape{(b)}] If $p|N$, then
\begin{align*}
\varphi_{\mathfrak{c}} (p, t) \ll_{\epsilon} \frac{b'(p,b_0)}{|\Gamma(\frac{1}{2}+it)|(pNb )^{ 1/2}}(N(1+|t|))^{\epsilon}.
\end{align*}
\item[\upshape{(c)}] If $e$ is a positive integer, then
\begin{align*}
\varphi_{\mathfrak{c}} (e^2, t) \ll_{\epsilon} \frac{b'e^{1+\epsilon}}{|\Gamma(\frac{1}{2}+it)|(Nb)^{\frac{1}{2} }}  (N(1+|t|))^{\epsilon}.
\end{align*}
\end{enumerate}
The bounds in (b) and (c) hold unconditionally with ineffective implied constants, or conditionally on GRH with effective implied constants.
\end{lem}
\begin{proof}
Theorem~3.4 of Kiral and Young~\cite{KY} with $\mathfrak{a}=1/N$ implies that $\varphi_{\mathfrak{c}} (n,t)$ is zero unless
\begin{equation}\label{eqn: KYThm3.4condition}
n = \frac{b'}{(b,N/b)} m
\end{equation}
for some integer $m$. Combining \eqref{eqn: KYThm3.4condition} and Equation (3.3) in \cite{KY}, we have
\begin{align}
\varphi_{\mathfrak{c}} (n, t) =
& \frac{\pi^{\frac{1}{2}+it}}{\Gamma(\frac{1}{2}+it)} n^{-\frac{1}{2}+it} \frac{(b,N/b)^{\frac{1}{2}+it}}{(Nb)^{\frac{1}{2}+it}}  \frac{b'}{(b,N/b)} S(\tfrac{m}{(m,(b,N/b))},0;b_0) \notag\\
& \times \sum_{\substack{d|m \\ (d,N/b)=1}} d^{-2it} \frac{1}{\varphi(\frac{(b,N/b)}{(m,(b,N/b))})} \sum_{\chi \bmod \frac{(b,N/b)}{(m,(b,N/b))} }  \frac{\chi(-\tfrac{m}{(m,(b,N/b))}\overline{b_0d^2}a) \tau(\overline{\chi}) }{L(1+2it,\overline{\chi^2}\chi_0)}, \label{eqn: KYThm3.4}
\end{align}
where $\chi_0$ is the principal character modulo $N/b$.

If $n=p$ with $p\nmid N$, then we must have $m=p$ in \eqref{eqn: KYThm3.4condition} because $b' | N.$ Thus $\varphi_{\mathfrak{c}}(p,t)=0$ unless \eqref{eqn: factorb} holds. In this case, $m=p$, $(m,(b,N/b))=1$ since $p\nmid N$, and
$$
S(\tfrac{m}{(m,(b,N/b))},0;b_0) = S(p,0;b_0)=\mu(b_0)
$$
by Kluyver's formula for the Ramanujan sum, and so \eqref{eqn: KYThm3.4} simplifies to give the expression for $\varphi_{\mathfrak{c}}(p,t)$ in (a). Similarly, if $n=e^2$ and \eqref{eqn: factorb} holds, then in \eqref{eqn: KYThm3.4} we have $m=e^2$ and $(m,(b,N/b))=(e^2,b')$, and \eqref{eqn: KYThm3.4} simplifies to the expression for $\varphi_{\mathfrak{c}}(e^2,t)$ in (a).

To prove (b), we may assume that \eqref{eqn: KYThm3.4condition} holds with $n=p$, since otherwise $\varphi_{\mathfrak{c}}(p,t)=0$ by Theorem~3.4 of \cite{KY}. This means either $m=p$ or $m=1$. If $m=p$, then \eqref{eqn: factorb} holds, and \eqref{eqn: KYThm3.4} with $n=p$ simplifies to
\begin{align*}
\varphi_{\mathfrak{c}} (p, t) =
& \frac{\pi^{\frac{1}{2}+it}}{\Gamma(\frac{1}{2}+it)} p^{-\frac{1}{2}+it} \frac{(b')^{\frac{1}{2}+it}}{(Nb)^{\frac{1}{2}+it}}  S(\tfrac{p}{(p,b')},0;b_0)  \frac{1}{\varphi(\frac{b'}{(p,b')})} \sum_{\chi \bmod \frac{b'}{(p,b')} }  \frac{\chi(-\tfrac{p}{(p,b')}\overline{b_0}a) \tau(\overline{\chi}) }{L(1+2it,\overline{\chi^2}\chi_0)} \\
& + \mathbf{1}_{(p,N/b)=1} \frac{\pi^{\frac{1}{2}+it}}{\Gamma(\frac{1}{2}+it)} p^{-\frac{1}{2}-it} \frac{(b')^{\frac{1}{2}+it}}{(Nb)^{\frac{1}{2}+it}}  S(p,0;b_0)  \frac{1}{\varphi(b')} \sum_{\chi \bmod b' }  \frac{\chi(-\overline{pb_0}a) \tau(\overline{\chi}) }{L(1+2it,\overline{\chi^2}\chi_0)},
\end{align*}
where $\mathbf{1}_{(p,N/b)=1}$ is $1$ if $(p,N/b)=1$ and is $0$ otherwise. We will use the following bounds to bound $\varphi_{\mathfrak {c}}(p, t).$  In fact, these bounds will also be useful later. 
\es{ \label{eqn:variousboundsforSection7}
S(\alpha,0;\gamma)&\ll (\alpha,\gamma) \ll \alpha; \\
\tau(\chi) &\ll \sqrt {c} , \ \ \ \ \ \ \textrm{where} \ \ \chi \  \textrm{is a character mod} \ c; \\
\sum_{d | \ell} 1 & \ll \ell^{\epsilon} ; \\
\frac{1}{L(1+2it,\overline{\chi^2}\chi_0)} &\ll (N(1+|t|))^{\epsilon}.
}
The first bound follows from the well-known bound for Ramanujan's sum, while the second bound is a well-known bound for the Gauss sum (see, e.g., \cite{Da}) . The last one follows from GRH (or Siegel's theorem, with an ineffective implied constant) (see, e.g., \cite{Da}). Using \eqref{eqn:variousboundsforSection7}, we deduce (b) for the case $m=p$. On the other hand, if $m=1$, then $(m,(b,N/b))=1$ and \eqref{eqn: KYThm3.4} simplifies to
\begin{align*}
\varphi_{\mathfrak{c}} (p, t) =
& \frac{\pi^{\frac{1}{2}+it}}{\Gamma(\frac{1}{2}+it)} p^{-\frac{1}{2}+it} \frac{(b,N/b)^{\frac{1}{2}+it}}{(Nb)^{\frac{1}{2}+it}}  \frac{b'}{(b,N/b)}  \frac{S(1,0;b_0)}{\varphi((b,N/b))} \sum_{\chi \bmod (b,N/b) }  \frac{\chi(-\overline{b_0}a) \tau(\overline{\chi}) }{L(1+2it,\overline{\chi^2}\chi_0)}.
\end{align*}
From this, the formula $S(1,0;b_0)=\mu(b_0)$ for Ramanujan's sum, and the bounds in \eqref{eqn:variousboundsforSection7}, we deduce that
\begin{align*}
\varphi_{\mathfrak{c}} (p, t) \ll_{\epsilon} \frac{b'}{|\Gamma(\frac{1}{2}+it)|(pNb )^{ 1/2}}(N(1+|t|))^{\epsilon},
\end{align*}
which implies (b).

To prove (c), take $n=e^2$ in \eqref{eqn: KYThm3.4} and use the bounds \eqref{eqn:variousboundsforSection7}. 

\end{proof}

For brevity, let $N = cL_1rd$. Proposition~\ref{prop:Ctn} will immediately follow from the following bounds. The first one is the contribution from $p | N$, and the second Proposition is the contribution 
from $p \nmid N.$

\begin{prop}\label{prop: p|N} Let
\begin{align*} 
\mathcal {CTN}_{p |N} &:= \sum_{c\geq 1} \sum_{p | N} \frac{ \log p}{\sqrt{p}} V\bfrac{p}{P}  \e{v\frac{p}{P}} \sum_{\mathfrak{c}} \int_{-\infty}^{\infty}  \frac{\sqrt{pe^2}}{\cosh(\pi t)} \overline{\varphi_{\mathfrak{c}}} (e^2, t)  \varphi_{\mathfrak{c} } (p, t) h_+ (t) \,dt.
\end{align*}
Then 
$$ \mathcal {CTN}_{p|N} \ll   (1+|u|)^2  Q^{\epsilon}  \frac{\sqrt{P}}{Q}.$$
\end{prop}

\begin{prop} \label{prop:contpnotdivideN} Let 
\begin{align*}
\mathcal {CTN}_{p \nmid N} &:= \sum_{c \geq 1}\sum_{p \nmid N} \frac{ \log p}{\sqrt{p}} V\bfrac{p}{P}  \e{v\frac{p}{P}} \sum_{\mathfrak{c}} \int_{-\infty}^{\infty}  \frac{\sqrt{pe^2}}{\cosh(\pi t)} \overline{\varphi_{\mathfrak{c}}} (e^2, t)  \varphi_{\mathfrak{c} } (p, t) h_+ (t) \,dt .
\end{align*}
Then
$$ \mathcal {CTN}_{p \nmid N} \ll (1 + |v|)^2(1 + |u|)^2  Q^{\epsilon} \left(P^{\frac 14 + \epsilon} + \frac{\sqrt P}{Q}\right). $$
\end{prop}

\subsection{Proof of Proposition \ref{prop: p|N} -- Contribution from $p | N$}

We write each $\mathfrak{c}$ as $(ab)^{-1}$ as stated in Lemma~\ref{lem:KiralYoungcor} to express the sum $\mathcal{CTN}_{p |N}$ as
\begin{align*}
\sum_{c\geq 1} 
 \sum_{b|N}\sideset{}{^*}\sum_{a \bmod (b,N/b)}  \sum_{p| N} \frac{ \log p}{\sqrt{p}} V\bfrac{p}{P}  \e{v\frac{p}{P}} \int_{-\infty}^{\infty}  \frac{\sqrt{pe^2}}{\cosh(\pi t)} \overline{\varphi_{\mathfrak{c}}} (e^2, t)  \varphi_{\mathfrak{c} } (p, t) h_+ (t) \,dt.
\end{align*}
We apply Lemma~\ref{lem:KiralYounglem} (b) and (c) and the identity $\Gamma(\frac{1}{2}+it)\Gamma(\frac{1}{2}-it)=\pi/\cosh(\pi t)$ to deduce that
\begin{align*}
& \sum_{c\geq 1}\sum_{b|N}\sideset{}{^*}\sum_{a \bmod (b,N/b)}  \sum_{p| N} \frac{ \log p}{\sqrt{p}} V\bfrac{p}{P}  \e{v\frac{p}{P}} \int_{-\infty}^{\infty}  \frac{\sqrt{pe^2}}{\cosh(\pi t)} \overline{\varphi_{\mathfrak{c}}} (e^2, t)  \varphi_{\mathfrak{c} } (p, t) h_+ (t) \,dt \\
& \ll \sum_{c\geq 1}\frac{e^{2+\epsilon}}{N^{1-\epsilon}} \sum_{b|N} \sideset{}{^*}\sum_{a \bmod (b,N/b)}  \sum_{p| N} \frac{\log p}{\sqrt{p}}V\bfrac{p}{P}\frac{(b')^2(p,b_0)}{ b} \int_{-\infty}^{\infty}   (1+|t|)^{\epsilon} |h_+ (t)| \,dt.
\end{align*}
Since $b=b_0b'$, it follows that
\begin{align*}
\frac{(b')^2(p,b_0)}{ b} \leq  \frac{bb'b_0}{b} = b.
\end{align*}
Moreover, we have
\begin{align*}
\sideset{}{^*}\sum_{a \bmod (b,N/b)} 1 \leq (b,N/b) \leq \frac{N}{b}.
\end{align*}
Hence, the above is
\begin{align*}
& \ll \sum_{c\geq 1} N^{\epsilon}e^{2+\epsilon} \sum_{b|N}  \sum_{p| N} \frac{\log p}{\sqrt{p}}V\bfrac{p}{P} \int_{-\infty}^{\infty}   (1+|t|)^{\epsilon} |h_+ (t)| \,dt \\
& \ll \sum_{c\geq 1} N^{\epsilon}e^{2+\epsilon} \int_{-\infty}^{\infty}   (1+|t|)^{\epsilon} |h_+ (t)| \,dt.
\end{align*}
Next, we use the bound for $h_+(t)$ from Lemma~\ref{lem:boundforhu} and recall from \eqref{eqn: Xdefinition} that $X = \frac{4\pi L_1L_2 \sqrt{Pe^2}}{cQ}$. Thus, for some $F<(1+|u|)( 1 + X),$ we have
\es{\label{eqn:CTN_p|N_finalbound}
\mathcal{CTN}_{p |N} 
& \ll  e^{2+\epsilon}Q^{\epsilon} \sum_{c} \min\left\{ X^{k-1}, \frac{1}{\sqrt X} \right\} \int_{-\infty}^{\infty}    (1+|t|)^{\epsilon} \frac{1 + |\log X|}{F^{1 - \epsilon}} \left( \frac{F}{1 + |t|}\right)^2 \,dt \\
& \ll (1+|u|)^2  Q^{\epsilon} \frac{\sqrt{P}}{Q},
}
as desired.

\subsection{Proof of Proposition \ref{prop:contpnotdivideN} -- Contribution from $p \nmid N$}
By expressing $\mathfrak{c}$ as $(ab)^{-1}$ (see Lemma~\ref{lem:KiralYoungcor}), we can write  
\begin{align*}
\mathcal {CTN}_{p \nmid N} 
&= \sum_{c \geq 1} \sum_{b|N}\sideset{}{^*}\sum_{a \bmod (b,N/b)}  \sum_{p \nmid N} \frac{ \log p}{\sqrt{p}} V\bfrac{p}{P}  \e{v\frac{p}{P}} \int_{-\infty}^{\infty}  \frac{\sqrt{pe^2}}{\cosh(\pi t)} \overline{\varphi_{\mathfrak{c}}} (e^2, t)  \varphi_{\mathfrak{c} } (p, t) h_+ (t) \,dt.
\end{align*}
By Lemma~\ref{lem:KiralYounglem}(a), the sum over $p$ here is zero unless \eqref{eqn: factorb} holds. Writing explicit expressions for $\varphi_{\mathfrak c}(p, t)$ and $\varphi_{\mathfrak c}(e^2, t)$ from Lemma~\ref{lem:KiralYounglem} (a),  we obtain that
\begin{align*}
\mathcal {CTN}_{p \nmid N} = E_1+E_2,
\end{align*}
where
\begin{align*}
E_1 : = & \sum_{c\geq 1} \sum_{\substack{b|N\\ b' = (b, N/b)}} \sideset{}{^*}\sum_{a \bmod (b,N/b)}  \sum_{p \nmid N} \frac{ \log p}{\sqrt{p}} V\bfrac{p}{P}  \e{v\frac{p}{P}} \int_{-\infty}^{\infty}  \frac{\sqrt{pe^2}}{\cosh(\pi t)} \frac{\pi^{\frac{1}{2}-it}}{\Gamma(\frac{1}{2}-it)} e^{-1-2it} \frac{1}{(Nb_0)^{\frac{1}{2}-it}} \\
& \times \overline{S(\tfrac{e^2}{(e^2,b')},0;b_0)}  \sum_{\substack{d|e^2 \\ (d,N/b)=1}} d^{2it} \frac{1}{\varphi(\frac{b'}{(e^2,b')})} \sum_{\chi \bmod \frac{b'}{(e^2,b')} }  \frac{\overline{\chi}(-\tfrac{e^2}{(e^2,b')}\overline{b_0d^2}a)  \overline{\tau( \overline{\chi} )  }}{L(1-2it, \chi^2 \chi_0)} \\
& \times \frac{\pi^{\frac{1}{2}+it}}{\Gamma(\frac{1}{2}+it)} \frac{\mu(b_0)}{(Nb_0)^{\frac{1}{2}+it}} p^{-\frac{1}{2}-it} \frac{1}{\varphi( b' )}  \sum_{\psi \bmod b' }  \frac{\psi(-\overline{p}\overline{b_0}a) \tau(\overline{\psi}) }{L(1+2it,\overline{\psi^2}\chi_0)}  h_+ (t) \,dt
\end{align*}
and $E_2$ is the same sum except with $\psi(\overline{p})p^{-it}$ replaced with $\psi(p)p^{it}$. From now on, we will focus on $E_1$ as $E_2$ can be treated similarly. Since $\Gamma(\frac{1}{2}+it)\Gamma(\frac{1}{2}-it)=\pi/\cosh(\pi t)$, $E_1$ simplifies to
\begin{align*}
E_1 = & \sum_{c\geq 1}\sum_{\substack{b|N\\ b' = (b, N/b)}} \frac{1}{Nb_0}\sideset{}{^*}\sum_{a \bmod (b,N/b)}  \sum_{p \nmid N}  \frac{\log p}{\sqrt{p}} V\bfrac{p}{P}  \e{v\frac{p}{P}} \int_{-\infty}^{\infty}     e^{-2it}p^{-it} \sum_{\substack{d|e^2 \\ (d,N/b)=1}} d^{2it} \\
& \times \overline{S(\tfrac{e^2}{(e^2,b')},0;b_0)}   \frac{1}{\varphi(\frac{b'}{(e^2,b')})} \sum_{\chi \bmod \frac{b'}{(e^2,b')} }  \frac{\overline{\chi}(-\tfrac{e^2}{(e^2,b')}\overline{b_0d^2}a) \overline{\tau( \overline{\chi} )  }}{L(1-2it, \chi^2 \chi_0)} \\
& \times  \mu(b_0)  \frac{1}{\varphi( b' )}  \sum_{\psi \bmod b' }  \frac{\psi(-\overline{p}\overline{b_0}a) \tau(\overline{\psi}) }{L(1+2it,\overline{\psi^2}\chi_0)}  h_+ (t) \,dt \\
&= E_{princ} + E_{non-princ},
\end{align*}
where $E_{princ}$ is the contribution from the principal characters $\psi$, and $E_{non-princ}$ is the contribution from the non-principal characters $\psi$. Proposition~\ref{prop:contpnotdivideN} will follow immediately from the following two Lemmas.
\begin{lem}\label{prop:principal}
Let the notation be as above. Then
$$E_{princ} \ll (1 + |v|)^2(1 + |u|)^2  Q^{\epsilon} \left(P^{\frac 14 + \epsilon} + \frac{\sqrt P}{Q}\right) .$$

\end{lem}

\begin{lem}\label{prop:nonprincipal} Let the notation be as above. 
Then 
$$E_{non-princ} \ll  (1+|u|)^2(1+|v|)^2  Q^{\epsilon} \frac{\sqrt{P}}{Q}.$$
\end{lem}

\subsection{Proof of Lemma \ref{prop:principal} -- The principal character}
Writing out $E_{princ}$ explicitly, we see that
\begin{align*}
E_{princ} &= \sum_{c \geq 1}\sum_{\substack{b|N\\ b' = (b, N/b)}} \frac{1}{Nb_0}\sideset{}{^*}\sum_{a \bmod (b,N/b)} \int_{-\infty}^{\infty}     e^{-2it}  \sum_{p \nmid N}  \frac{\log p}{\sqrt{p}} V\bfrac{p}{P}  \e{v\frac{p}{P}} p^{-it} \sum_{\substack{d|e^2 \\ (d,N/b)=1}} d^{2it} \\
& \times \overline{S(\tfrac{e^2}{(e^2,b')},0;b_0)}   \frac{1}{\varphi(\frac{b'}{(e^2,b')})} \sum_{\chi \bmod \frac{b'}{(e^2,b')} }  \frac{\overline{\chi}(-\tfrac{e^2}{(e^2,b')}\overline{b_0d^2}a)  \tau( \chi )  }{L(1-2it, \chi^2 \chi_0)} \\
& \times  \mu(b_0)  \frac{1}{\varphi( b' )}  \frac{\mu(b')}{L(1+2it,\psi_0^2\chi_0)}  h_+ (t) \,dt,
\end{align*}
where $\psi_0$ is the principal character modulo $b'$. We start by dealing with the sum over $p.$ 

\begin{lem} \label{lem:sumoverpforPrincipal}
Let $V_1(x) = V(x)\e{vx}$.  We have that
\begin{equation*}
\sum_{p\nmid N}  \frac{\log p}{p^{1/2+it}} V\bfrac{p}{P}  \e{v\frac{p}{P}} 
= P^{1/2 - it} \tilde{V_1}(1/2-it) + O(P^{\epsilon} (1+|v|)^2\log N).
\end{equation*}
\end{lem}
\begin{proof}
We use Mellin inversion to write
\begin{align*}
\sum_{p\nmid N}  \frac{\log p}{p^{1/2+it}} V\bfrac{p}{P}  \e{v\frac{p}{P}} = \frac{1}{2\pi i} \int_{(\frac{1}{2}+\epsilon)} P^{w} \widetilde{V}_1(w) \sum_{p\nmid N}  \frac{\log p}{p^{1/2+it+w}} \,dw,
\end{align*}
where $\widetilde{V}_1$ is the Mellin transform of $V_1$. Since $V_1^{(l)}(x) \ll (1+|v|)^l$ for all $l\ge 0$, $\widetilde{V_1}(s) \ll \bfrac{1+|v|}{1+|s|}^A$ for any $A \ge 0$. Now the Euler product formula for $\zeta(s)$ implies
\begin{align*}
\sum_{p\nmid N}  \frac{\log p}{p^{1/2+it+w}}
& = \sum_{p }  \frac{\log p}{p^{1/2+it+w}} -\sum_{p|N }  \frac{\log p}{p^{1/2+it+w}} \\
& = \sum_{n }  \frac{\Lambda(n)}{n^{1/2+it+w}} - \sum_{m=2}^{\infty}\sum_p \frac{\log p}{(p^m)^{1/2+it+w}}-\sum_{p|N }  \frac{\log p}{p^{1/2+it+w}} \\
& = -\frac{\zeta'}{\zeta}(\tfrac{1}{2}+it+w) + F(s),
\end{align*}
say, where $F(s) \ll \log N$ uniformly in the region Re$(w)\geq \epsilon$ and for all $t$.

Shifting contours to $\textrm{Re}(w)  = \epsilon$ gives the result upon picking up the pole at $w = 1/2 - it$ and using the bound $\widetilde{V_1}(s) \ll \bfrac{1+|v|}{1+|s|}^2$ to bound the integral along  $\textrm{Re}(w)  = \epsilon$.
\end{proof}

For clarity, we record the following bound for $h_+(z)$, where $z$ is a complex number. 

\begin{lem} \label{lem:boundh+}Recall that 
\begin{align*}
h_u(\xi)  = J_{k-1}(\xi)W\bfrac{\xi}{X} \e{u\xi}, 
\end{align*}
where $X = \frac{L_1L_2\sqrt{Pe^2}}{cQ}$ and $W$ is compactly supported in (a, b), where $a > 0.$ Then we have for real $\gamma$ and real $\delta \neq 0$, where $|\delta| < \frac 12$, that
    \begin{equation*}
        h_+(\gamma + i \delta) \ll (1 + |u|)  \min\left\{X^{k-1-2|\delta| }, \frac 1{\sqrt X}\right\}.
    \end{equation*}
\end{lem}

\begin{proof} We will slightly modify the proof of Lemma 7.1 in \cite{DI} and show it here for completeness. Write $z = \gamma + i\delta.$ We have
\begin{align*}
    h_+(z) &= \frac{2\pi i}{\sinh(\pi z)} \int_0^{\infty} (J_{2iz} (x) - J_{-2iz}(x)) h_u(x) \frac{dx}{x} \\
    &= -8 \int_0^{\infty} \int_0^{\infty} \cos(x\cosh(y)) \cos(2zy) h_u(x) \> dy \> \frac{dx}{x}.
\end{align*}
Consider the integral over $x$. Integrating by parts, we see that
\es{\label{eqn:intcoshu} \int_0^{\infty} \cos(x\cosh(y))  \frac{h_u(x)}{x} \> dx = -\int_0^{\infty} \left( \frac{h_u(x)}{x}\right)' \frac{\sin(x\cosh(y))}{\cosh(y)} \> dx. }

As in the proof of Lemma~\ref{lem:boundforhu}, we consider two cases depending on how large $X$ is. When $X \ll 1$, we write 
$$ J_{k-1}(\xi) = \sum_{\ell = 0}^{\infty} (-1)^{\ell} \frac{(\xi/2)^{2\ell + k - 1}}{\ell ! (\ell + k - 1)!}.$$
Therefore, for $aX < x < bX,$
$$ \left( \frac{h_u(x)}{x}\right)' \ll   X^{k - 1}  \left( (1 + |u|) +  \frac 1X\right)  \frac{1}{X} \ll   X^{k - 3}  \left( (1 + |u|)\right).   $$
We write $\frac{\sin(x\cosh(y))}{X\cosh(y)} \ll \min\left(1, \frac{e^{-y}}{X} \right)$, so that \eqref{eqn:intcoshu} is bounded by 
$$ X^{k-1} \left( 1 + |u| \right) \min \left\{ 1,  \frac {e^{-y}}{X}\right\} . $$
 From the above, the assumption $X \ll 1$, and the fact that $\cos(2zy) \ll e^{2|\delta|y}$, we arrive at
 $$ h_+(z) \ll \int_0^{\infty} e^{2|\delta| y} X^{k-1} (1 + |u|)\min \left\{ 1,  \frac {e^{-y}}X\right\} \> dy \ll X^{k-1} (1 + |u|) X^{-2|\delta|}. $$

When $X \gg 1$, from Lemma~\ref{jbessel}, we may write
\est{ J_{k-1}(\xi) =\frac{1}{ \sqrt{2\pi \xi}}\left(W_k(\xi) \e{\frac{\xi}{2\pi}-\frac{k}{4}+\frac{1}{8}} + \overline{W}_k(\xi) \e{-\frac{\xi}{2\pi}+\frac{k}{4}-\frac{1}{8} }\right),} where $W_k^{(j)}(x)\ll_{j,k} x^{-j}.$ 
Therefore for $aX < x < bX,$
$$ \left( \frac{h_u(x)}{x}\right)' \ll   \frac{1}{\sqrt X}  \left( (1 + |u|) +  \frac 1X\right)  \frac{1}{X}.   $$
Thus \eqref{eqn:intcoshu} is bounded by 
$$ \sqrt X ( 1 + |u|) \min \left\{ 1,  \frac {e^{-y}}{X}\right\} . $$
 From the above, the assumption $X \gg 1$, and the fact that $\cos(2zy) \ll e^{2|\delta|y}$, we deduce that
 $$ h_+(z) \ll \int_0^{\infty} e^{2|\delta| y} \sqrt X (1 + |u|)\min \left\{ 1,  \frac {e^{-y}}X\right\} \> dy \ll \frac 1{\sqrt X} (1 + |u|). $$
Combining both cases $X\ll 1$ and $X\gg1 $, we obtain the desired bound. 
\end{proof}

We now bound the integral over $t$ in the following lemma.
\begin{lem} \label{lem:intovertforPrinc}
    Let
    \begin{align*}
        I = \int_{-\infty}^{\infty}     e^{-2it}  d^{2it} \frac{1}{L(1-2it, \chi^2 \chi_0) L(1+2it,\psi_0^2\chi_0)}  h_+ (t) \tilde{V_1}(1/2-it) P^{1/2 - it} dt.
    \end{align*}

    We have that

    \begin{align*}
        I \ll  q_0^{\epsilon}\left( \frac ed\right)^{- \frac 12 + 2\epsilon} P^{\frac 14 + \epsilon} (1 + |u|)  \min\left\{X^{k- \frac 32 + 2\epsilon}, \frac 1{\sqrt X}\right\}  (1 + |v|)^2 , 
    \end{align*}
    where $q_0$ is the conductor of $\chi^2 \chi_0.$
\end{lem}

\begin{proof}
    Let $s = 1/2 - it$, so that $t = i(s-1/2)$.  We have that
    \begin{align*}
        I = i \int_{(1/2)}     \bfrac{e}{d}^{2s-1} \frac{1}{L(2s, \chi^2 \chi_0) L(2-2s,\psi_0^2\chi_0)}  h_+ (i(s-1/2)) \tilde{V_1}(s) P^{s} ds.
    \end{align*}
We move the line of integration to Re$(s) = \frac 14 + \epsilon$. Note that if $s = \sigma + ir$, then $h_+(i(s-1/2)) = h_+(-r +i(\sigma-1/2)). $ We use the bound for $h_+(z)$ in Lemma~\ref{lem:boundh+} and the bound
$$L\left(\frac 12 + \epsilon + it, \chi^2 \chi_0\right)^{-1} \ll (q_0 (|t| + 1))^{\epsilon} 
$$ 
from Theorem~1 in \cite{Chirre} to deduce that
\est{I &\ll \left( \frac ed\right)^{- \frac 12 + 2\epsilon} P^{\frac 14 + \epsilon} (1 + |u|)  \min\left\{X^{k- \frac 32 + 2\epsilon }, \frac 1{\sqrt X}\right\} \\
& \ \  \ \ \ \ \ \ \ \ \hskip 1in  \times \int_{-\infty}^{\infty}  \tilde{V_1}\left(\frac 14 + \epsilon + it\right) (q_0(|t| + 1))^{\epsilon} \> dt  \\ 
&\ll \left( \frac ed\right)^{- \frac 12 + 2\epsilon} P^{\frac 14 + \epsilon} (1 + |u|)  \min\left\{X^{k- \frac 32 + 2\epsilon }, \frac 1{\sqrt X}\right\} q_0^{\epsilon} (1 + |v|)^{2}, }
as desired.
\end{proof}

We are now ready to bound $E_{princ}$ in Lemma~\ref{prop:principal}. By Lemma~\ref{lem:sumoverpforPrincipal}, we have 
\begin{align*}
E_{princ} &= \sum_{c \geq 1}\sum_{\substack{b|N\\ \text{\eqref{eqn: factorb}}}} \frac{1}{Nb_0}\sideset{}{^*}\sum_{a \bmod (b,N/b)}   \int_{-\infty}^{\infty}     e^{-2it}  \left[P^{\frac 12 - it} \tilde V_1\left( \frac12 - it\right) + O\left(P^{\epsilon} (1 + |v|)^2 \log N\right)\right] \\
& \times \sum_{\substack{d|e^2 \\ (d,N/b)=1}} d^{2it}\overline{S(\tfrac{e^2}{(e^2,b')},0;b_0)}   \frac{1}{\varphi(\frac{b'}{(e^2,b')})} \sum_{\chi \bmod \frac{b'}{(e^2,b')} }  \frac{\overline{\chi}(-\tfrac{e^2}{(e^2,b')}\overline{b_0d^2}a)  \tau( \chi )  }{L(1-2it, \chi^2 \chi_0)} \\
& \times  \mu(b_0)  \frac{1}{\varphi( b' )}  \frac{\mu(b')}{L(1+2it,\psi_0^2\chi_0)}  h_+ (t) \,dt. 
\end{align*}
First, we consider the contribution from the error term $O\left(P^{\epsilon} (1 + |v|)^2 \log N\right)$, say $E_{princ, err}$. We use the bound for $h_+(t)$ from Lemma~\ref{lem:boundforhu}~(1) and the bounds in \eqref{eqn:variousboundsforSection7} to deduce that
\es{ \label{eqn:Eprinc_err}E_{princ, err} &\ll (1 + |v|)^2 e^{2 + \epsilon} \sum_{c \geq 1} \sum_{b | N} \sideset{}{^*}\sum_{a \bmod (b,N/b)}(PN)^{\epsilon}  \frac{1}{Nb_0} F^{1 + \epsilon} \min\left\{ X^{k - 1}, \frac 1{\sqrt X} \right\} (1 + |\log X|)  \frac{\sqrt{b'}}{\phi(b')} \\ &\ll (1 + |u|)^2 (1 + |v|)^2 Q^{\epsilon} \frac{\sqrt P}{Q}}
upon recalling that $b' = (b, N/b)$, $e \ll Q^{\epsilon}$, and $X = \frac{L_1L_2 \sqrt {Pe^2}}{cQ}.$
For the main contribution (call it $E_{princ, main}$), we use Lemma~\ref{lem:intovertforPrinc} to bound the integral over $t$ and obtain that 
\est{E_{princ, main} &\ll  \sum_{c \geq 1} \sum_{b | N} \sideset{}{^*}\sum_{a \bmod (b,N/b)} \frac{1}{Nb_0} \sum_{d | e^2} q_0^{\epsilon}\left( \frac ed\right)^{- \frac 12 + 2\epsilon} P^{\frac 14 + \epsilon} (1 + |u|)\mathcal  F(c)  (1 + |v|)^2   e^2 \frac{\sqrt{b'}}{\phi(b')} \\
&\ll (1 + |v|)^2(1 + |u|)^2 P^{\frac 14 + \epsilon} Q^{\epsilon} \sum_{c} \sum_{b | N} \frac {(b, N/b)^{3/2}}{Nb} \mathcal F(c), }
where we let $\mathcal F(c) := \min\left\{X^{k- \frac 32 + 2\epsilon }, \frac 1{\sqrt X}\right\}$ and we recall that $b' = (b, N/b)$ and $b_0 = b/(b, N/b)$. Moreover, recall that $N = crdL_1 := c \alpha,$ where $\alpha = rdL_1 \ll Q^{\epsilon}.$ Considering the sum over $c$, we have 
\est{\sum_{c} \sum_{b | N} \frac {(b, N/b)^{3/2}}{c \alpha b}  \mathcal F(c) &\leq \sum_{c} \sum_{b_1 | \alpha} \sum_{b_2 | c} \frac{(b_1b_2, \frac{c\alpha}{b_1b_2})^{3/2}}{c\alpha b_1b_2} \mathcal F(c) \\
&\ll \sum_{c} \sum_{b_2 | c} \frac{(b_2, c/b_2)^{3/2}}{cb_2} \mathcal F(c)\sum_{b_1 | \alpha} \frac{b_1^{3/2} (\alpha/ b_1)^{3/2}}{b_1 \alpha} \\
&\ll \alpha^{1/2 + \epsilon} \sum_{b_2} \frac{1}{b_2^2} \sum_c \frac{(b_2, c)^{3/2}}{c} \mathcal F(cb_2) \\
&\ll  \alpha^{1/2 + \epsilon} \sum_{b_2} \frac{1}{b_2^2} \sum_{\ell | b_2} \ell^{3/2} \sum_c \frac{1}{c\ell} \mathcal F(c\ell b_2)  \\
&\ll \alpha^{1/2 + \epsilon} \sum_{b_2} \frac{1}{b_2^2} \sum_{\ell | b_2} \ell^{1/2} \ll \alpha^{1/2 + \epsilon} \ll Q^{\epsilon}.}
Thus \es{\label{eqn:Eprin_main} E_{princ, main} \ll  (1 + |v|)^2(1 + |u|)^2 P^{\frac 14 + \epsilon} Q^{\epsilon}.}
Combining \eqref{eqn:Eprinc_err} and \eqref{eqn:Eprin_main}, we arrive at the lemma.

\subsection{Proof of Lemma \ref{prop:nonprincipal} -- Non-principal characters}
Recall that 
\begin{align*}
 E_{non-princ} &= \sum_{c\geq 1} \sum_{\substack{b|N\\ \text{\eqref{eqn: factorb}}}} \frac{1}{Nb_0}\sideset{}{^*}\sum_{a \bmod (b,N/b)} \int_{-\infty}^{\infty}   \sum_{p\nmid N}  \frac{\log p}{p^{1/2+it}} V\bfrac{p}{P}  \e{v\frac{p}{P}}    e^{-2it} \sum_{\substack{d|e^2 \\ (d,N/b)=1}} d^{2it} \\
& \times \overline{S(\tfrac{e^2}{(e^2,b')},0;b_0)}   \frac{1}{\varphi(\frac{b'}{(e^2,b')})} \sum_{\chi \bmod \frac{b'}{(e^2,b')} }  \frac{\overline{\chi}(-\tfrac{e^2}{(e^2,b')}\overline{b_0d^2}a)  \tau( \chi )  }{L(1-2it, \chi^2 \chi_0)} \\
& \times  \mu(b_0)  \frac{1}{\varphi( b' )}  \sum_{\substack{\psi \bmod b'\\ \psi\neq \psi_0} }  \frac{\psi(-\overline{p}\overline{b_0}a) \tau(\overline{\psi}) }{L(1+2it,\overline{\psi^2}\chi_0)}  h_+ (t) \,dt.
\end{align*}
Since $\psi$ is a non-principal character modulo $b'$, it follows from Lemma~\ref{lem:CLee3.5} that if $V_0$ is a smooth function supported on a compact subinterval of $(0,\infty)$, then
\begin{align*}
 \sum_{p \nmid N} \frac{\psi(\overline{p}) \log p}{p^{\frac{1}{2}+it}} V_0\left( \frac{p}{P}\right) \ll  \log^2 (b'+|t|) + \log N.
\end{align*}
Thus, by the argument in the proof of Lemma~\ref{lem:primesumMaassbdd}, we have
\begin{align*}
\sum_{p \nmid N}  \frac{\psi(\overline{p})\log p}{p^{1/2+it}} V\bfrac{p}{P}  \e{v\frac{p}{P}} \ll  (1+|v|)^2 [\log^2 (b'+|t|) + \log N].
\end{align*}
It follows from this and the bounds in \eqref{eqn:variousboundsforSection7} that 
\begin{align*}
E_{non-princ}&\ll (1+|v|)^2 e^{2+\epsilon} 
\sum_{c\geq 1} \sum_{\substack{b|N\\ \text{\eqref{eqn: factorb}}}} \frac{b'}{N^{1-\epsilon} b_0}\sideset{}{^*}\sum_{a \bmod (b,N/b)} \int_{-\infty}^{\infty}    (1+|t|)^{\epsilon} | h_+ (t) | \,dt \\
&\ll (1+|v|)^2 e^{2+\epsilon} 
\sum_{c\geq 1} \sum_{b|N} \frac{(b,N/b)^2}{N^{1-\epsilon} b_0} \int_{-\infty}^{\infty}    (1+|t|)^{\epsilon} | h_+ (t) | \,dt,
\end{align*}
where we recall that $b'=(b,N/b)$. We then use $e \ll Q^{\epsilon}$, the bound for $h_+(t)$ from Lemma~\ref{lem:boundforhu}, and  
\begin{align*}
(b,N/b)^2 \leq b\cdot \frac{N}{b} = N,
\end{align*}
and proceed with the same argument as in \eqref{eqn:CTN_p|N_finalbound} to obtain the desired bound.



\section*{Acknowledgement}
S.B. acknowledges support from a NSF grant DMS-1854398. V.C. acknowledges support from a Simons Travel Grant for Mathematicians and NSF grant DMS-2101806. X.L. acknowledges support from Simons Grants 524790 and 962494 and NSF grant DMS-2302672. Parts of this work were done while S.B. was visiting Kansas State University. He would like to express his gratitude to Kansas State University for the invitation and the pleasant atmosphere for working.




\begin{thebibliography}{9999}

\bibitem{AL}  A.O.L. Atkin and J. Lehner, {\it  Hecke operators on $\Gamma_0(m)$}. Math. Ann. 185, 134-160 (1970). https://doi.org/10.1007/BF01359701.

\bibitem{BBDDM} O. Barrett, P. Burkhardt, J. DeWitt, R. Dorward, and S.J. Miller, {\it One-level density for holomorphic cusp forms of arbitrary level.} Res. Number Theory, 3: Art. 25, 21,2017.


\bibitem{BHM} V. Blomer, G. Harcos and P. Michel, {\it Bounds for modular $L$-functions in the level aspect}, Ann. Sci. Ecole Norm. Sup. (4) 40 (2007), no. 5, 697-740.

\bibitem{BM}  V. Blomer and D. Mili\'cevi\'c. {\it The second moment of twisted modular L-functions}. Geom. Funct. Anal., 25(2) (2015), 453 - 516.

\bibitem{CLee} V. Chandee and Y. Lee, {\it $n$-level density of the low lying zeros of primitive Dirichlet $L$-functions}, Adv. Math. (2020) 369, available online at https://doi.org/10.1016/j.aim.2020.107185.


\bibitem{CKLL} V. Chandee, K. Klinger-Logan, and X. Li, {\it Pair correlation of zeros of $\Gamma_1(q)$ $L$-functions}, Math. Z. 302 (2022), 219 - 258. 

\bibitem{CL}  V. Chandee and X. Li, {\it The eighth moment of Dirichlet L-functions.} Adv. Math. {\bf 259} (2014), 339-375.

\bibitem{CLMR} V. Chandee, X. Li, K. Matom\"aki, and M. Radziwi\l\l, {\it The eighth moment of Dirichlet $L$-functions II}, submitted.


\bibitem{Chirre} A. Chirre, {\it  A Note on Entire $L$-Functions.} Bull Braz Math Soc, New Series 50, 67-93 (2019). https://doi.org/10.1007/s00574-018-0092-x


\bibitem{CIS}   J.B. Conrey, H. Iwaniec and K. Soundararajan {\it Asymptotic large sieve	}, 	arXiv:1105.1176
\bibitem{CIS2} J.B. Conrey, H. Iwaniec and K. Soundararajan, {\it The sixth power moment of Dirichlet $L$-functions},  Geometric and Functional Analysis {\bf 22} (2012), issue 5, 1257-1288.
\bibitem{CIS3} J.B. Conrey, H. Iwaniec, and K. Soundararajan, {\it Critical zeros of Dirichlet L-functions.} J. Reine Angew. Math. 681 (2013), 175-198. 	

\bibitem{Da} H. Davenport, {\it Multiplicative Number Theory}, vol.74, Springer-Verlag (GTM), New York, 2000.

\bibitem{DI} J.-M. Deshouillers and H. Iwaniec, {\it Kloosterman sums and Fourier coefficients of cusp forms}, Invent. Math. 70 (1982), 219-288.

\bibitem{DPR} S.Drappeau, K. Pratt, and M. Radziwi\l\l, \ 	{\it One-level density estimates for Dirichlet $L$-functions with extended support}, Algebra and Number theory {\bf 17:4} (2023), 805-830.


\bibitem{DFI} W. Duke, J. Friedlander and H. Iwaniec, {\it The subconvexity problem for {A}rtin {$L$}-functions}, Invent. Math. 149 (2002), p. 489--577.

\bibitem{FM} D. Fiorilli and S. J. Miller, {\it Surpassing the ratios conjecture in the 1-level density of Dirichlet $L$-functions}, Algebra Number Theory 9 (2015), no. 1, 13-52.





\bibitem{GR} I. S. Gradshteyn and I. M. Ryzhik, {\it Table of Integrals, Series, and Products}, Edited by A.Jeffrey and D. Zwillinger. Academic Press, New York, 7th edition, 2007.


\bibitem{HR} C. P. Hughes and Z. Rudnick, {\it Linear Statistics of Low-Lying Zeros of $L$-Functions}, Q. J. Math {\bf 54:3} (2003), 309 - 333.



\bibitem{Iwaniec} H. Iwaniec, {\it Topics in classical automorphic 
forms}, Graduate Studies in Mathematics, vol. 17 (American Mathematical Society, Providence, RI, 1997).

\bibitem{IK} H. Iwaniec and E. Kowalski, {\it Analytic Number Theory}, vol. 53, American Mathematical Society Colloquium Publications, Rhode Island, 2004.

\bibitem{ILS} H. Iwaniec, W. Luo and P. Sarnak {\it Low lying zeros of families of L-functions.} Inst. Hautes Etudes Sci. Publ. Math. No. 91 (2000), 55-131 (2001).


\bibitem{KaSa} N. Katz and P. Sarnak, {\it Random matrices, Frobenius eigenvalues, and monodromy.} American Mathematical Society Colloquium Publications, 45. American Mathematical Society, Providence, RI, 1999.



\bibitem{KimS} H. Kim and P. Sarnak, Appendix 2 in {\it Functoriality for the exterior square of $GL_4$ and the symmetric fourth of $GL_2$}, J. Amer. Math. Soc. {\bf 16} (2003), 139-183.

\bibitem{KY} E.M. Kiral and M. Young, {\it  Kloosterman sums and Fourier coefficients of Eisenstein series}, Ramanujan J. 49 (2019), no. 2, 391-409.



\bibitem{GMar} G. Martin, {\it Dimensions of the spaces of cusp forms and newforms on $\Gamma_0(N)$ and $\Gamma_1(N)$.}  J. Number Theory 112 (2005), no.2, 298-331.








\bibitem{Ng} M-H. Ng, The basis for space of cusp forms and Petersson trace formula.  Master of Philosophy thesis at The University of Hong Kong, 2012.  Available at http://hub.hku.hk/handle/10722/174338.












\bibitem{Pe} I. Petrow, {\it Bounds for Traces of Hecke Operators and Applications to Modular and Elliptic Curves Over a Finite Field.} Algebra \& Number Theory. 12(10) pp. 2471-2498 (2018).


\bibitem{rouymi} D. Rouymi, {\it Formules de trace et non-annulation de fonctions $L$ automorphes au niveau $\mathfrak{p}^{\nu}$}, Acta Arith. {\bf 147}, (2011), 1--32.







\bibitem{Watt}  G. N. Watson, {\it A treatise on the theory of Bessel functions}, Cambridge University Press, Cambridge 1944.
\end{thebibliography}
\end{document}